\documentclass[11pt]{article}
\usepackage{arxiv}
\usepackage[utf8]{inputenc} 
\usepackage[T1]{fontenc}    
\usepackage{amsmath,amssymb,amsthm,mathtools,bm,bbm}
\usepackage{enumitem}
\usepackage{framed} 
\usepackage{xcolor}
\usepackage{hyperref}

\usepackage{url}            
\usepackage{booktabs}       
\usepackage{amsfonts}       
\usepackage{nicefrac}       
\usepackage{microtype}      
\usepackage{lipsum}
\usepackage{graphicx}
\usepackage{indentfirst}
\usepackage{makecell}
\usepackage{caption}
\usepackage{subcaption}
\usepackage{float}
\usepackage{adjustbox}
\usepackage[left]{lineno}
\usepackage{hyphenat}
\usepackage{microtype}
\usepackage{hyphenat}
\hypersetup{colorlinks=true,linkcolor=blue,citecolor=teal,urlcolor=magenta}

\theoremstyle{plain}
\newtheorem{theorem}{Theorem}[section]
\newtheorem{lemma}[theorem]{Lemma}
\newtheorem{proposition}[theorem]{Proposition}
\newtheorem{corollary}[theorem]{Corollary}
\theoremstyle{definition}
\newtheorem{definition}[theorem]{Definition}

\theoremstyle{remark}
\newtheorem{remark}[theorem]{Remark}

\newcommand{\R}{\mathbb{R}}

\newcommand{\norm}[1]{\left\lVert #1 \right\rVert}
\newcommand{\ip}[2]{\left\langle #1,#2 \right\rangle}

\newcommand{\grad}{\nabla}
\newcommand{\cG}{\mathcal{G}}

\newcommand{\cS}{\mathcal{S}}
\newcommand{\cP}{\mathcal{P}}

\title{\Large \bf Complexity Bounds for Smooth Multiobjective Optimization}
\author{
 Phillipe R. Sampaio\thanks{Email:
 \texttt{phillipe.rodriguessampaio@bnpparibas.com}} \\
  BNP Paribas Cardif\\
  Nanterre, France \\
}

\begin{document}
\setlength\parindent{15pt}
\maketitle

\begin{abstract}
We study the oracle complexity of finding $\varepsilon$-Pareto stationary points in smooth multiobjective optimization with $m$ objectives. Progress is measured by the Pareto stationarity gap $\cG(x)$, the norm of the best convex combination of objective gradients. Our analysis relies on a non-degenerate lifting that embeds hard single-objective instances into MOO instances with distinct objectives and non-singleton Pareto fronts while preserving lower bounds on $\cG$. We establish: (i) in the $\mu$-strongly convex case, any span first-order method has worst-case linear convergence no faster than $\exp(-\Theta(T/\sqrt{\kappa}))$ after $T$ oracle calls, yielding $\Theta(\sqrt{\kappa}\log(1/\varepsilon))$ iterations and matching accelerated upper bounds; (ii) in the convex case, an $\Omega(1/T)$ min-iterate lower bound for oblivious one-step methods and a universal last-iterate lower bound $\Omega(1/T^2)$ for oblivious span methods via polynomial-degree arguments, and we further show this latter bound is loose (for general adaptive methods) by importing geometric lower bounds to obtain an $\Omega(1/T)$ min-iterate lower bound for general adaptive first-order methods; (iii) in the nonconvex case with $L$-Lipschitz gradients, an $\Omega(\sqrt{L}/(T+1))$-type lower bound on $\cG$ (tight in order), implying $\Omega(1/\varepsilon^2)$ iterations to reach $\cG(x)\le\varepsilon$ up to natural scaling. 
\end{abstract}


\section{Introduction}
Multiobjective optimization (MOO) provides a mathematical framework for decision-making problems involving multiple, often conflicting, performance criteria. Given a vector-valued objective function {\scriptsize $F(x) = (f_1(x), \dots, f_m(x))$}, the goal is to find a point $x$ that offers the best possible trade-off among the individual objectives $f_i$. Since a single point that simultaneously minimizes all objectives typically does not exist, the central solution concept is that of \emph{Pareto optimality}. A point is Pareto optimal if no other point can improve one objective without degrading at least one other.

First-order iterative methods have become the tool of choice for solving large-scale MOO problems, particularly in machine learning \cite{SenerKoltun18}, and they are increasingly deployed in engineering and finance — for example, distributed-gradient schemes for multi-objective AC optimal power flow \cite{MoghadamSuratgarHesamzadehNikravesh2022}, adjoint-based gradient methods for reservoir well-control under risk and uncertainty \cite{LiuReynolds2016SPE}, and gradient-descent formulations for multi-objective portfolio selection \cite{OlivaVenturaLagoFernandez2025}. These methods rely on an oracle that returns function values and gradients. A key challenge is to define a tractable computational goal. While computing an exact Pareto-optimal point is often impractical, a widely adopted necessary condition is \emph{Pareto criticality} \cite{FliegeSvaiter00}, which is equivalent to the \emph{Pareto stationarity} notion adopted here. A point $x$ is Pareto stationary if the convex hull of its gradients contains the origin. The proximity to this condition is measured by the Pareto stationarity gap $\cG(x)$ (see Definition~\ref{def:pareto-gap} in Section~\ref{sec:prelim}). Finding a point $x$ with $\mathcal{G}(x) \le \varepsilon$ is a standard objective for many contemporary MOO algorithms \cite{Desideri12}.

Although numerous algorithms have been proposed with convergence-rate guarantees (upper bounds), the fundamental limits of what is achievable have remained largely unexplored. Without \emph{oracle lower bounds}, we cannot know if existing algorithms are optimal. In single-objective optimization, this story is complete: the seminal work of Nemirovski and Yudin \cite{NY83} established lower bounds that were later matched by accelerated methods \cite{Nesterov04}, proving their optimality. For MOO, however, a corresponding theory of oracle complexity has been absent.

\paragraph{Why scalarization does not trivialize MOO}
It may appear that multiobjective complexity reduces immediately to the single-objective case by fixing a scalarization $f_\lambda=\sum_i \lambda_i f_i$ and applying known convergence results. Indeed, all classical rates for smooth convex or strongly convex minimization (e.g., gradient descent $O(LR/\sqrt{T})$, AGD $O(LR/T)$, and linear convergence under strong convexity) apply verbatim to any fixed scalarization. However, the multiobjective goal is different: the central complexity measure is the Pareto stationarity gap $\cG(x)$ which requires considering all convex combinations of gradients simultaneously. A point minimizing a fixed $f_\lambda$ need not minimize $\mathcal{G}$, and the algorithm classes we study (oblivious one-step vs.\ oblivious span vs.\ general adaptive first-order methods) differ in their ability to exploit information across objectives. Consequently, while scalarization provides useful upper bounds, the oracle complexity of Pareto stationarity requires a dedicated analysis, which is the focus of this paper.

\paragraph{Contributions} This paper closes this fundamental gap by providing a nuanced information-based complexity analysis that distinguishes between different classes of first-order algorithms. Our main contributions are as follows:
\begin{itemize}[leftmargin=1.5em,itemsep=0.35em]
    \item \textbf{Robust non-degenerate lifting.} We show our lower bounds are not artifacts of degenerate problems. We construct MOO instances with distinct objectives and non-singleton Pareto fronts that preserve lower bounds on $\cG$ from embedded scalar instances.

    \item \textbf{Tight bounds for strongly convex MOO:} We establish a tight linear convergence lower bound. Any span first-order method can be forced to have a stationarity gap of at least $\mu R \left(\frac{\sqrt{\kappa}-1}{\sqrt{\kappa}+1}\right)^T$, implying an iteration complexity of $\Omega(\sqrt{\kappa} \log(\mu R/\varepsilon))$. This is matched by accelerated gradient methods on a fixed scalarization.

    \item \textbf{Lower bound for oblivious one-step convex MOO:} For the class of \emph{oblivious one-step gradient methods}—which includes gradient descent with a pre-scheduled step-size sequence—we prove an $\Omega(LR/(T+1))$ lower bound on the (min-iterate) stationarity gap. Accelerated gradient descent (AGD), which is not in this class but is an \emph{oblivious span} method, achieves an $\mathcal{O}(LR/(T+1))$ upper rate on a fixed scalarization. Thus the $\Omega(1/T)$ rate is matched by oblivious span methods, while within the oblivious one-step class the best known general upper bound is $\mathcal{O}(LR/\sqrt{T})$.

    \item \textbf{Polynomial-degree lower bound for oblivious span methods and geometric lower bound for adaptive methods:} Within the \emph{oblivious span} model, we establish a universal \emph{last-iterate} polynomial-degree lower bound $\Omega(1/T^2)$ via Markov-type extremal inequalities. For general deterministic adaptive first-order methods, we prove a stronger \emph{geometric} (zero-chain) $\Omega(1/T)$ \emph{min-iterate} lower bound under a bounded scalarization gap, obtained by lifting the stationary-point lower bounds of Carmon--Duchi--Hinder--Sidford \cite{CarmonDuchiHinderSidford2021}.

    \item \textbf{Nonconvex extension with Lipschitz gradients.} Using the same lifting on the stationary-point lower bounds of Carmon--Duchi--Hinder--Sidford \cite{CarmonDuchiHinderSidford2020}, we obtain a nonconvex MOO lower bound under $L$-Lipschitz gradients of $\Omega(\sqrt{L\Delta/(T+1)})$ (equivalently $\Omega(L\Delta/\varepsilon^2)$ iterations) to reach $\cG(x)\le\varepsilon$.
\end{itemize}
Our methodology combines classical polynomial-based lower-bound techniques in the spirit of the resisting-oracle framework \cite{NY83} with a non-degenerate lifting reduction that transfers sharp single-objective stationary-point lower bounds (e.g., \cite{CarmonDuchiHinderSidford2021}) to multiobjective Pareto stationarity. This yields a unified oracle-complexity perspective for smooth MOO and clarifies how algorithmic restrictions (oblivious one-step vs.\ oblivious span vs.\ general adaptive first-order methods) and iterate selection (min vs.\ last) shape the achievable rates.

\paragraph{Organization} Section~\ref{sec:related} reviews related work. Section~\ref{sec:prelim} fixes notation, defines $\cG$, and formalizes oracle models and algorithm classes. Section~\ref{sec:lower} presents the lifting construction, then derives strongly convex, convex, and nonconvex lower bounds by combining polynomial extremal arguments and geometric stationary-point lower bounds. Section~\ref{sec:upper} provides matching upper bounds via AGD on fixed scalarizations and compares rates against the lower bounds. Section~\ref{sec:conclusions} concludes with implications and open questions. The appendices contain the polynomial extremal derivations (Chebyshev and Markov) used to analyze span and oblivious methods on hard quadratic instances.

\section{Related Work}\label{sec:related}

\paragraph{Oracle complexity and first-order limits}
The information-based complexity program for single-objective convex optimization established sharp first-order lower bounds together with matching optimal methods \cite{NY83,Nesterov04}. In \cite{NY83}, these lower bounds are derived by reducing the worst-case error to an extremal polynomial problem, solved exactly by (scaled) Chebyshev polynomials; see \cite{Rivlin74} for the underlying Chebyshev minimax and related inequalities. Recent analyses include Performance Estimation Problem (PEP) frameworks that deliver tight worst-case guarantees for many first-order methods \cite{DroriTeboulle14}. Beyond the classical oracle model, a line of work formalizes oblivious and structured first-order algorithms: Arjevani and Shamir \cite{ArjevaniShamir16ICML} introduced the oblivious framework with dimension-free lower bounds for $L$-smooth convex and $L$-smooth $\mu$-strongly convex problems; Arjevani, Shalev-Shwartz, and Shamir \cite{ArjevaniShalevOhad16JMLR} developed the $p$-SCLI framework (covering $p{=}1$ one-step methods) with matching bounds; and Arjevani and Shamir \cite{ArjevaniShamir16NeurIPS} extended dimension-free iteration lower bounds to finite-sum settings, covering variance-reduced families. Complementary results for randomized first-order algorithms appear in Woodworth and Srebro \cite{WoodworthSrebro17}. More recently, Carmon, Duchi, Hinder, and Sidford \cite{CarmonDuchiHinderSidford2021} established sharp lower bounds for finding first-order stationary points in smooth convex and nonconvex optimization, clarifying the fundamental $1/T$ (convex) and $1/\sqrt{T}$ (nonconvex) barriers in gradient-norm complexity under standard first-order oracles. A comprehensive and recent treatment of finite-sum lower bounds is given by Han, Xie, and Zhang \cite{HanXieZhang24}.

\paragraph{Multiobjective optimization algorithms}
Classical foundations and broad context are covered in the monograph by Miettinen \cite{Miettinen1998} and the survey of Marler and Arora \cite{MarlerArora2004Survey}. For descent-type methods, the steepest–descent family beginning with \cite{FliegeSvaiter00} and including cone-ordered variants and refined convergence analyses \cite{DrummondSvaiter2005,CocchiLiuzziLucidiSciandrone2020} provides global convergence to Pareto-stationary points. Worst-case complexity guarantees of order $\mathcal{O}(\varepsilon^{-2})$ for trust-region and regularization frameworks in unconstrained MOO were established in \cite{GrapigliaYuanYuan2015}. Complexity aspects for gradient descent in smooth MOO are analyzed in \cite{FliegeVazVicente2019}. Universal nonmonotone line-search frameworks for non-convex MOO with convex constraints have also been proposed \cite{PinheiroGrapiglia2025Universal}. When progress is tracked through standard merit functions or stationarity surrogates, the resulting upper bounds are typically sublinear (e.g., $\mathcal{O}(1/k)$) in smooth convex settings \cite{TanabeFukudaYamashita2023OptLett, GoncalvesGoncalvesMelo2025FWRates}. Second-order frameworks — Newton-type and SQP — offer globalization with fast local behavior: globally convergent Newton methods for MOO with fast (often superlinear) local rates under standard regularity are developed in \cite{GoncalvesLimaPrudente2022}, while SQP-based schemes for constrained problems are given in \cite{FliegeVaz2016MOSQP}. Conditional-gradient (Frank--Wolfe) methods have been developed for vector/multiobjective optimization with global convergence to weakly efficient (Pareto-stationary) solutions under cone-convexity assumptions \cite{ChenYangZhao2023}. Away-step and related variants can accelerate convergence and, under additional geometric or curvature assumptions, deliver faster (sometimes linear) rates \cite{GoncalvesGoncalvesMelo2024AwayFW,GoncalvesGoncalvesMelo2025FWRates}, with adaptive generalized conditional-gradient methods extending this line \cite{GebrieFukuda2025}. Accelerated first-order schemes tailored to the multiobjective setting continue to mature: an accelerated proximal-gradient method with $\mathcal{O}(1/k^2)$ decrease for an appropriate merit measure is established in \cite{TanabeFukudaYamashita23}, and recent work clarifies monotonicity properties of multiobjective APG \cite{NishimuraFukudaYamashita2024}. Complementary practical gradient-based schemes and stationarity criteria, such as multiple-gradient descent, are surveyed in \cite{Desideri12}, and links to multi-task learning further motivate stationarity-based goals and algorithm design \cite{SenerKoltun18}.

\paragraph{Our position}
While scalarization-based upper bounds are classical, a systematic oracle-complexity theory for \emph{Pareto stationarity} that (i) distinguishes oblivious one-step from both oblivious span methods and general adaptive first-order methods, (ii) provides robust non-degenerate hard MOO instances, and (iii) connects polynomial-degree lower bounds to geometric stationary-point lower bounds, has been missing. We supply such a framework and import sharp scalar stationary-point lower bounds into MOO via an explicit lifting reduction.

\section{Preliminaries}\label{sec:prelim}

\subsection{Notation}
All vectors lie in $\mathbb{R}^d$ and matrices in $\mathbb{R}^{d\times d}$. 
The Euclidean inner product and norm are $\langle x,y\rangle$ and $\|x\|$; for a matrix $A$, 
$\|A\|_{\mathrm{op}}$ is the operator norm, $\sigma(A)$ its spectrum, $A\succeq 0$ (resp.\ $A\succ 0$) means positive (semi)definite, and $I$ is the identity; $(\cdot)^\top$ denotes transpose.
For a set $S$, $\operatorname{conv}(S)$ is its convex hull, $\operatorname{span}(S)$ its linear span, and $\operatorname{dist}(x,S):=\inf_{y\in S}\|x-y\|$ its distance. 
For a matrix $A$, $\operatorname{range}(A)$ and $\ker(A)$ denote its range and kernel.
We write $\mathbb{N}_0:=\{0,1,2,\ldots\}$ and $\mathbb{N}:=\{1,2,\ldots\}$.

For $m\ge2$ objectives $F(x)=(f_1(x),\dots,f_m(x))$, 
the unit simplex is 
\[
\Delta^m:=\{\lambda\in\R^m:\ \lambda_i\ge0,\ \sum_{i=1}^m\lambda_i=1\}.
\]
For $\lambda\in\Delta^m$, the scalarization is $f_\lambda(x):=\sum_{i=1}^m \lambda_i f_i(x)$ and $x^\star_\lambda\in\arg\min_x f_\lambda(x)$ denotes an arbitrary minimizer.

In convex settings we use the scalarization value gap at initialization
\[
\Delta_\lambda(x^{(0)}):= f_\lambda(x^{(0)})-\inf_{x} f_\lambda(x).
\]

When the Pareto set $\cP$ is well-defined (convex case; see Proposition~\ref{prop:convex-iff}), we write
\[
R:=\mathrm{dist}(x^{(0)},\cP).
\]
In the nonconvex extension we will instead refer to the set of Pareto-stationary points $\cS:=\{x:\cG(x)=0\}$ when needed.

Iterates are $x^{(t)}$ for $t\in\mathbb{N}_0$. $T\in\mathbb{N}$ is the number of oracle calls/iterations.

The condition number is $\kappa:=L/\mu$ when strong convexity applies. We occasionally write $\rho:=\frac{\sqrt{\kappa}+1}{\sqrt{\kappa}-1}$ in spectral formulas.

For quadratics $g(x)=\tfrac12 x^\top H x-b^\top x$, we use the spectral variable $\zeta$ and real polynomials $p_t$ (with $\deg p_t\le t$, $p_t(0)=1$) to express errors via $x^{(t)}-x^\star=p_t(H)\bigl(x^{(0)}-x^\star\bigr)$. Chebyshev polynomials of the first kind are denoted $T_t(\cdot)$.

Standard Landau symbols $O(\cdot),\ \Omega(\cdot),\ \Theta(\cdot)$ are used; $\log$ denotes the natural logarithm.

For nonnegative quantities $A$ and $B$, we write $A \asymp B$ (equivalence up to universal constants) if there exist absolute constants $c_1,c_2>0$ (independent of problem parameters such as $L,\Delta,R,T,m,d$) such that $c_1 B \le A \le c_2 B$.
We use this notation only when the implicit constants are universal; otherwise we state explicitly which parameters they may depend on.

The vectors $(e_i)_{i=1}^m$ are the canonical basis of $\mathbb{R}^m$.

We consider unconstrained optimization problems in $\R^d$. A differentiable function $f: \R^d \to \R$ is \textbf{$L$-smooth} if its gradient is $L$-Lipschitz continuous. It is \textbf{$\mu$-strongly convex} if $f(y) \ge f(x) + \ip{\nabla f(x)}{y-x} + \frac{\mu}{2}\norm{y-x}^2$ for all $x,y$.

\subsection{Problem formulation}\label{subsec:smooth-convex-prob}

We consider unconstrained multiobjective problems
\begin{equation}\label{eq:main-prob}
\min_{x\in\R^d} F(x) := (f_1(x),\ldots,f_m(x)).
\end{equation}
where each $f_i: \R^d \to \R$ is continuously differentiable. Convexity and strong convexity will be imposed in the relevant sections.

\subsection{Pareto optimality and first-order conditions}\label{subsec:pareto-optimality}

We now review the standard optimality notions for (possibly non-convex) MOO problems (see \cite{Miettinen1998}) and clarify their relationship with the Pareto stationarity gap used throughout the paper. 

\smallskip
\begin{definition}[Dominance]
Given $x,y\in\R^d$, we say that $y$ \emph{dominates} $x$ if
\[
f_i(y)\le f_i(x)\ \ \text{for all } i=1,\dots,m,\quad \text{and}\quad
f_j(y)< f_j(x)\ \ \text{for at least one } j.
\]
\end{definition}

\smallskip
Based on this, a point is considered optimal if no other point in the domain dominates it.

\smallskip
\begin{definition}[Pareto optimality (strong)]
A point $x^\star\in\R^d$ is \emph{Pareto optimal} if there is no $y\in\R^d$ that dominates $x^\star$; equivalently, there is no $y$ with
\[
f_i(y)\le f_i(x^\star)\ \ \text{for all } i,\quad \text{and}\quad
f_j(y)< f_j(x^\star)\ \ \text{for some } j.
\]
\end{definition}

\smallskip
A related, slightly weaker condition is often useful for theoretical analysis and algorithm design:

\smallskip
\begin{definition}[Weak Pareto optimality]
A point $x^\star\in\R^d$ is \emph{weakly Pareto optimal} if there is no $y\in\R^d$ such that
\[
f_i(y)< f_i(x^\star)\ \ \text{for all } i=1,\dots,m.
\]
\end{definition}

\smallskip
To develop iterative, gradient-based methods, we need a computable, first-order necessary condition for optimality. This leads to the concept of Pareto criticality.

\smallskip
\begin{definition}[Pareto criticality]\label{def:pareto-criticality}
Following \cite{FliegeSvaiter00}, a point $x\in\R^d$ is called \emph{Pareto critical} if there is no common strict descent direction, i.e.,
\[
\not\exists\, v\in\R^d \quad\text{such that}\quad \nabla f_i(x)^\top v < 0 \quad \text{for all } i=1,\dots,m.
\]
Equivalently, with the Jacobian $JF(x)\in\R^{m\times d}$ whose $i$-th row is $\nabla f_i(x)^\top$, this reads
\[
\operatorname{range}(JF(x))\cap(-\R_{++})^m=\varnothing,
\]
where $\R_{++}$ is the set of (strictly) positive real numbers.
\end{definition}

\smallskip
This geometric condition is equivalent to the convex hull of the gradients containing the origin. The proximity to this state is quantified by the following gap function.

\smallskip
\begin{definition}[Pareto stationarity gap]\label{def:pareto-gap}
The Pareto stationarity gap at $x\in\R^d$ is
\[
\cG(x) \;:=\; \min_{\lambda\in\Delta^m} \Big\| \sum_{i=1}^m \lambda_i\,\nabla f_i(x) \Big\|,
\qquad
\Delta^m := \big\{\lambda\in\R^m:\ \lambda_i\ge 0,\ \sum_{i=1}^m \lambda_i=1\big\}.
\]
We call $x$ \emph{$\varepsilon$-Pareto stationary} if $\cG(x)\le \varepsilon$; the case $\varepsilon=0$ corresponds to \emph{Pareto stationarity}.
\end{definition}

\smallskip
\paragraph{Relation to Pareto merit functions}
Besides stationarity gaps, the MOO literature also uses merit (value-gap) functions. Assume that each $f_i$ is convex and differentiable. For instance, consider the Pareto merit function
\[
u_0(x)\;:=\;\sup_{z\in\mathbb{R}^d}\min_{i=1,\dots,m}\bigl(f_i(x)-f_i(z)\bigr),
\]
which vanishes at weakly Pareto optima and is positive otherwise (see, e.g., \cite{TanabeFukudaYamashita23}).

Let $f_\lambda(x):=\sum_{i=1}^m \lambda_i f_i(x)$ for $\lambda\in\Delta^m$, and assume
$\arg\min f_\lambda\neq\emptyset$ for all $\lambda\in\Delta^m$; pick any minimizer
$x_\lambda^\star\in\arg\min f_\lambda$ and write $f_\lambda^\star:=f_\lambda(x_\lambda^\star)$.
We can upper bound $u_0(x)$ by scalarization gaps in three simple steps.
Fix $x$ and $z$ and apply the identity
$\min_i a_i=\min_{\lambda\in\Delta^m}\sum_i \lambda_i a_i$ with $a_i=f_i(x)-f_i(z)$, to obtain
\[
\min_{i=1,\dots,m}\bigl(f_i(x)-f_i(z)\bigr)
=\min_{\lambda\in\Delta^m}\sum_{i=1}^m \lambda_i\bigl(f_i(x)-f_i(z)\bigr)
=\min_{\lambda\in\Delta^m}\bigl(f_\lambda(x)-f_\lambda(z)\bigr).
\]
Substituting into the definition of $u_0$ yields
\[
u_0(x)=\sup_{z\in\mathbb{R}^d}\min_{\lambda\in\Delta^m}\bigl(f_\lambda(x)-f_\lambda(z)\bigr).
\]
Next, we use the elementary min--max inequality 
\[
\sup_z\min_\lambda \Phi(z,\lambda)\le \min_\lambda\sup_z \Phi(z,\lambda)
\]
with $\Phi(z,\lambda)=f_\lambda(x)-f_\lambda(z)$ to obtain
\[
u_0(x)\;\le\;\min_{\lambda\in\Delta^m}\sup_{z\in\mathbb{R}^d}\bigl(f_\lambda(x)-f_\lambda(z)\bigr).
\]
Finally, for each fixed $\lambda$, since $f_\lambda(x)$ does not depend on $z$,
\[
\sup_{z\in\mathbb{R}^d}\bigl(f_\lambda(x)-f_\lambda(z)\bigr)
=f_\lambda(x)-\inf_{z\in\mathbb{R}^d} f_\lambda(z)
=f_\lambda(x)-f_\lambda^\star,
\]
and therefore
\[
u_0(x)\;\le\;\min_{\lambda\in\Delta^m}\bigl(f_\lambda(x)-f_\lambda^\star\bigr).
\]

By convexity, for each $\lambda$,
\[
f_\lambda(x)-f_\lambda^\star
\le \langle \nabla f_\lambda(x),\,x-x_\lambda^\star\rangle
\le \|\nabla f_\lambda(x)\|\,\|x-x_\lambda^\star\|.
\]
Therefore,
\[
u_0(x)\;\le\;\bar R(x)\,\mathcal G(x),
\qquad
\bar R(x):=\sup_{\lambda\in\Delta^m}\|x-x_\lambda^\star\|.
\]
Moreover, if each $f_i$ is $\mu$-strongly convex, then so is every $f_\lambda$,
and the Polyak--\L ojasiewicz inequality yields
$f_\lambda(x)-f_\lambda^\star\le \frac{1}{2\mu}\|\nabla f_\lambda(x)\|^2$; hence
\[
u_0(x)\;\le\;\frac{1}{2\mu}\,\mathcal G(x)^2.
\]
In particular, under a boundedness condition ensuring $\bar R(x)<\infty$ on the region of interest (or under strong convexity),
any convergence guarantee stated in terms of the stationarity gap $\mathcal G(x)$ immediately implies a corresponding bound
for the merit function $u_0(x)$.

\smallskip
\paragraph{Equivalence of stationarity notions}
The next lemma makes precise that \emph{Pareto criticality} is equivalent to \emph{Pareto stationarity} $\cG(x)=0$.

\smallskip
\begin{lemma}[No common descent $\Longleftrightarrow$ convex-hull stationarity]\label{lem:equiv-stationarity}
For $C^1$ objectives $(f_i)$, the following are equivalent at a point $x$:
\begin{enumerate}[label=(\alph*),leftmargin=2em]
\item \textup{(Pareto criticality)} There is no $v\in\R^d$ with $\nabla f_i(x)^\top v<0$ for all $i$.
\item \textup{(Convex-hull stationarity)} There exists $\lambda\in\Delta^m$ such that 
\[
\sum_{i=1}^m \lambda_i\,\nabla f_i(x)=0
\]
(equivalently, $\cG(x)=0$).
\end{enumerate}
\end{lemma}

\begin{proof}
Apply Gordan’s theorem of the alternative to $A=JF(x)$: exactly one of $Av<0$ or $A^\top y=0$ with $y\ge 0$, $y\neq 0$ holds. The negation of the former yields a nonzero $y\ge 0$ with $JF(x)^\top y=0$; normalizing $y$ to the simplex gives (b). Conversely, (b) rules out any $v$ with $JF(x)v<0$ because $\sum_i \lambda_i\,\nabla f_i(x)^\top v$ would be a convex combination of strictly negative numbers.
\end{proof}

\smallskip
\paragraph{Necessary and sufficient conditions}
We now record the basic relationships between optimality and stationarity in the unconstrained setting.

\smallskip
\begin{proposition}[General $C^1$ case: necessity]\label{prop:necessary-general}
If $x^\star$ is weakly Pareto optimal (for not-necessarily convex $C^1$ objectives), then $x^\star$ is Pareto stationary; equivalently $\cG(x^\star)=0$.
\end{proposition}

\begin{proof}
At a weakly Pareto optimal point there is no direction that strictly decreases all objectives. By Lemma~\ref{lem:equiv-stationarity}, this is equivalent to the existence of $\lambda\in\Delta^m$ with $\sum_i \lambda_i\,\nabla f_i(x^\star)=0$, hence $\cG(x^\star)=0$.
\end{proof}

\begin{proposition}[Convex $C^1$ case: characterization of weak Pareto optima]\label{prop:convex-iff}
Suppose each $f_i$ is convex and $C^1$. Then, for $x^\star\in\R^d$, the following are equivalent:
\begin{enumerate}[label=(\roman*),leftmargin=2em]
\item $x^\star$ is weakly Pareto optimal;
\item $x^\star$ is Pareto stationary, i.e., $\cG(x^\star)=0$;
\item there exists $\lambda\in\Delta^m$ such that $x^\star\in\arg\min_x f_\lambda(x)$ with $f_\lambda:=\sum_{i=1}^m \lambda_i f_i$.
\end{enumerate}
\end{proposition}

\begin{proof}
(i)$\Rightarrow$(ii) is Proposition~\ref{prop:necessary-general}. (ii)$\Rightarrow$(iii): if $\sum_i \lambda_i\nabla f_i(x^\star)=0$ for some $\lambda\in\Delta^m$, then $x^\star$ satisfies the first-order optimality condition for the convex function $f_\lambda$, so $x^\star$ minimizes $f_\lambda$. (iii)$\Rightarrow$(i): if $x^\star$ minimizes $f_\lambda$ and there were $y$ with $f_i(y)<f_i(x^\star)$ for all $i$, then $f_\lambda(y)<f_\lambda(x^\star)$, a contradiction.
\end{proof}

\begin{remark}
Under convexity, Proposition~\ref{prop:convex-iff} shows that $\cG(\cdot)$ is an \emph{exact} first-order optimality measure for weak Pareto optimality: $\cG(x^\star)=0$ if and only if $x^\star$ is weakly Pareto optimal. In the general (nonconvex) $C^1$ case, $\cG(x^\star)=0$ remains a \emph{necessary} condition but is not sufficient, as usual for first-order stationarity.
\end{remark}

\subsection{Oracle Model and Algorithm Classes}

We formally define the classes of first-order algorithms considered in this work. The broadest class contains nearly all modern iterative methods.

\begin{definition}[Span First-Order Methods]
An algorithm is a \emph{span first-order method} if its iterates satisfy
\[
x^{(t)} \;\in\; x^{(0)} + \mathrm{span}\{\nabla f_i(x^{(j)}):\ i=1,\dots,m,\ j=0,\dots,t-1\}.
\]
This general class allows for adaptive step sizes and momentum terms that depend on the entire history of observed gradients.
\end{definition}

\begin{definition}[Oblivious Span First-Order Methods]\label{def:oblivious-span}
A span first-order method is \emph{oblivious} if, for each $t\ge 1$, there exist coefficients
$\{\beta^{(t)}_{i,j}\}_{i=1,\ldots,m;\,j=0,\ldots,t-1}$ that are fixed in advance (possibly as a function of known
problem parameters such as $L,\mu,m,T$) and do \emph{not} depend on the realized oracle replies, such that
\[
x^{(t)} \;=\; x^{(0)} + \sum_{j=0}^{t-1}\sum_{i=1}^{m}\beta^{(t)}_{i,j}\,\nabla f_i\bigl(x^{(j)}\bigr).
\]
Equivalently, the method satisfies the span property and uses only \emph{pre-scheduled} (instance-independent)
linear combination coefficients.
\end{definition}

\begin{remark}
The class of oblivious span methods includes accelerated schemes with fixed parameter schedules such as Nesterov's accelerated gradient method on a fixed scalarization.
\end{remark}

For our sharpest convex lower bound, we consider a more restricted, non-adaptive class.

\begin{definition}[Oblivious One-Step Gradient Methods]
Fix $\lambda\in\Delta^m$ and define $f_\lambda(x)=\sum_{i=1}^m \lambda_i f_i(x)$. An \emph{oblivious one-step gradient method} is an algorithm that generates iterates by
\[
x^{(t+1)} = x^{(t)} - \alpha_t \,\nabla f_\lambda(x^{(t)}), \qquad t=0,1,2,\dots,
\]
with pre-scheduled step sizes $\{\alpha_t\}$ satisfying $0 \le \alpha_t \le 1/L$ for all $t$ that may depend only on known problem parameters (e.g., \(L\), \(\mu\)), but not on the oracle’s feedback. For a quadratic \(g(x)=\tfrac{1}{2}x^\top H x - b^\top x\), the error satisfies \(x^{(t)}-x^\star = p_t(H)(x^{(0)}-x^\star)\) with
\[
p_t(\zeta) \;=\; \prod_{k=0}^{t-1} \bigl(1-\alpha_k \zeta\bigr), \qquad p_t(0)=1.
\]
\end{definition}

\begin{remark}
The oblivious one-step class models standard gradient descent with a fixed step schedule. It does not include methods with momentum, such as  AGD or Polyak's heavy-ball method.
\end{remark}

\begin{lemma}[Krylov Representation for Span Methods on Quadratics]\label{lem:krylov}
Let $g:\mathbb{R}^d\to\mathbb{R}$ be the quadratic
\[
g(x)=\tfrac12\,x^\top H x - b^\top x,
\]
where $H\in\mathbb{R}^{d\times d}$ is symmetric and $b\in\operatorname{range}(H)$, so that $g$ attains its minimum and there exists $x^\star\in\arg\min g$ with $H x^\star=b$. Consider any iterative algorithm whose iterates $\{x^{(t)}\}_{t\ge0}$ satisfy the \emph{span property}
\begin{equation}\label{eq:span-property}
x^{(t)} \in x^{(0)}+\operatorname{span}\{\nabla g(x^{(0)}),\ldots,\nabla g(x^{(t-1)})\}
\quad\text{for all }t\ge1.
\end{equation}
Then, for every $t\ge0$, there exists a real polynomial $p_t$ of degree at most $t$ with $p_t(0)=1$ such that
\[
x^{(t)}-x^\star \;=\; p_t(H)\,\bigl(x^{(0)}-x^\star\bigr).
\]
Moreover, if the method is an \emph{oblivious span} method (Definition~\ref{def:oblivious-span}) when applied to $g$,
then for each $t$ the polynomial $p_t$ depends only on the method's fixed coefficients (and on $t$), and is therefore
independent of the specific quadratic instance $(H,b)$.

Furthermore, if the method is an oblivious one-step gradient method on $g$, i.e.,
\[
x^{(t+1)}=x^{(t)}-\alpha_t\,\nabla g(x^{(t)})\quad(t\ge0)
\]
for some pre-scheduled stepsizes $\{\alpha_t\}_{t\ge0}$, then
\[
p_t(\zeta)=\prod_{k=0}^{t-1}\bigl(1-\alpha_k\,\zeta\bigr)\qquad(t\ge1),\qquad p_0\equiv 1.
\]
The same conclusions hold when $g$ is any fixed scalarization of a quadratic multiobjective instance.
\end{lemma}

\begin{proof}
Fix any minimizer $x^\star$ with $H x^\star=b$; such a point exists by the assumption $b\in\operatorname{range}(H)$. For every $x$, we then have
\begin{equation}\label{eq:grad-translation}
\nabla g(x)=Hx-b=H\,(x-x^\star).
\end{equation}
Let $e^{(t)}:=x^{(t)}-x^\star$ and $e^{(0)}:=x^{(0)}-x^\star$. We prove by induction on $t$ that
\begin{equation}\label{eq:krylov-claim}
e^{(t)} \;=\; p_t(H)\,e^{(0)} \quad\text{for some polynomial } p_t \text{ with }\deg p_t\le t \text{ and } p_t(0)=1.
\end{equation}

\emph{Base case $t=0$.} Trivially, $e^{(0)}=I\,e^{(0)}$, so \eqref{eq:krylov-claim} holds with $p_0\equiv 1$.

\emph{Inductive step.} Assume \eqref{eq:krylov-claim} holds for all indices up to $t$. By the span property \eqref{eq:span-property}, there exist scalars $\{\beta_j^{(t)}\}_{j=0}^t$ such that
\[
x^{(t+1)} \;=\; x^{(0)} + \sum_{j=0}^{t} \beta_j^{(t)}\,\nabla g\bigl(x^{(j)}\bigr).
\]
(If the method is oblivious span, we take $\{\beta_j^{(t)}\}$ to be exactly the pre-scheduled coefficients specified
by the method.)
Subtracting $x^\star$ and using \eqref{eq:grad-translation} gives
\begin{equation}\label{eq:etplus1}
e^{(t+1)} \;=\; e^{(0)} + \sum_{j=0}^{t} \beta_j^{(t)}\,H\,e^{(j)}.
\end{equation}
By the induction hypothesis, for each $0\le j\le t$ there exists a polynomial $p_j$ with $\deg p_j\le j$ and $p_j(0)=1$ such that $e^{(j)}=p_j(H)\,e^{(0)}$. Substituting into \eqref{eq:etplus1} yields
\[
e^{(t+1)} \;=\; \Bigl(I + H\,\sum_{j=0}^{t} \beta_j^{(t)}\,p_j(H)\Bigr)\,e^{(0)}.
\]
Define the polynomial
\[
q_t(\zeta)\;:=\;\sum_{j=0}^{t} \beta_j^{(t)}\,p_j(\zeta),
\qquad\text{so that}\qquad
\deg q_t \le \max_{0\le j\le t}\deg p_j \le t,
\]
and set
\[
p_{t+1}(\zeta)\;:=\;1+\zeta\,q_t(\zeta).
\]
Then $\deg p_{t+1}\le t+1$, $p_{t+1}(0)=1$, and
\[
e^{(t+1)} \;=\; p_{t+1}(H)\,e^{(0)}.
\]
This completes the induction and proves \eqref{eq:krylov-claim} for all $t\ge0$.

\medskip
\emph{Obliviousness and instance-independence.}
If the method is an oblivious span method (Definition~\ref{def:oblivious-span}), then the coefficients
$\{\beta_j^{(t)}\}_{j=0}^t$ in \eqref{eq:etplus1} are fixed in advance (they may depend on $t$ and on known problem
parameters, but not on the realized oracle replies). Hence the recursion that constructs $p_{t+1}$ from
$(p_0,\ldots,p_t)$ uses only these fixed coefficients. Therefore, for each $t$, the resulting polynomial $p_t$ is
determined solely by the method's fixed coefficients (and by $t$), and in particular does not depend on $(H,b)$.

It remains to establish the stated product form for oblivious one-step methods. In that case,
\[
e^{(t+1)} \;=\; x^{(t)}-x^\star-\alpha_t\,\nabla g(x^{(t)})
\;=\; \bigl(I-\alpha_t H\bigr)\,e^{(t)} \qquad(t\ge0),
\]
using \eqref{eq:grad-translation}. Iterating this linear recurrence and recalling $e^{(0)}=p_0(H)\,e^{(0)}$ with $p_0\equiv 1$ gives
\[
e^{(t)} \;=\; \Bigl(\prod_{k=0}^{t-1}\bigl(I-\alpha_k H\bigr)\Bigr)\,e^{(0)}
\;=\; p_t(H)\,e^{(0)},
\]
with $p_t(\zeta)=\prod_{k=0}^{t-1}(1-\alpha_k \zeta)$, which has degree $t$ and satisfies $p_t(0)=1$.

\medskip
\emph{Independence of the choice of minimizer.} If $H$ is singular, the minimizer may not be unique: any $x^\star$ with $Hx^\star=b$ differs from another minimizer by a vector in $\ker(H)$. Let $x^\star{}'$ be another minimizer and write $x^\star{}'=x^\star+z$ with $Hz=0$. Then $e'^{(0)}=x^{(0)}-x^\star{}'=e^{(0)}-z$ and, by the proven representation,
\[
x^{(t)}-x^\star{}' \;=\; p_t(H)\,e'^{(0)} \;=\; p_t(H)\,e^{(0)} - p_t(0)\,z.
\]
Since $x^{(t)}-x^\star{}'=(x^{(t)}-x^\star)-z$, consistency for all $z\in\ker(H)$ requires and is ensured by $p_t(0)=1$. Thus the representation is well defined independently of the chosen minimizer.
\end{proof}

\section{Oracle Complexity Lower Bounds}\label{sec:lower}

\subsection{A Non-Degenerate Lifting Construction}\label{sec:non-degenerate}
A common concern with oracle lower bounds is that the hard instances may be degenerate from a multiobjective perspective—for example, the objectives may coincide up to constants, or the Pareto set may collapse to a singleton—so that the lower bound reflects a scalar artifact rather than a genuinely multiobjective phenomenon. 
To rule this out, we introduce a simple but robust \emph{lifting} device: it embeds an arbitrary scalar function $g$ defined on a subspace $V$ into an $m$-objective instance on $\R^d=V\oplus W$ by adding objective-dependent quadratic terms on a complementary subspace $W$. 
The resulting objectives are distinct, the Pareto set is non-singleton, and the Pareto stationarity gap $\cG$ admits an explicit decomposition showing that it dominates the scalar stationarity measure $\|\nabla g\|$. 
This ``gap domination'' property is the mechanism that transfers first-order lower bounds from the scalar problem to the multiobjective setting (including nonconvex $g$, since convexity is not used in the domination argument).

\begin{theorem}[Non-degenerate lifting and stationarity-gap domination]\label{thm:lifting}
Fix integers $m\ge2$ and $d\ge1$. Let $\R^d=V\oplus W$ be an orthogonal direct sum with $\dim(W)\ge m-1$, and pick $m$ affinely independent points $a_1,\dots,a_m\in W$. 
Let $g:V\to\R$ be $C^1$ with $L_g$-Lipschitz gradient, and fix $\gamma>0$. Define for $i=1,\dots,m$:
\begin{equation}\label{eq:lift-def}
f_i(x):=g(x_V)+\frac{\gamma}{2}\|x_W-a_i\|^2,\qquad x=x_V+x_W.
\end{equation}
Then:
\begin{enumerate}[label=(\roman*),leftmargin=2em]
\item The objectives $\{f_i\}_{i=1}^m$ are distinct.
\item Each $f_i$ is $L$-smooth with $L=\max\{L_g,\gamma\}$. If $g$ is convex (resp.\ $\mu$-strongly convex) and $\gamma\ge0$ (resp.\ $\gamma\ge\mu$), then each $f_i$ is convex (resp.\ $\mu$-strongly convex).
\item For all $x\in\R^d$,
\begin{equation}\label{eq:gap-formula}
\cG(x)^2=\|\grad g(x_V)\|^2+\gamma^2\,\mathrm{dist}\!\left(x_W,\mathrm{conv}\{a_1,\dots,a_m\}\right)^2,
\end{equation}
in particular $\cG(x)\ge \|\grad g(x_V)\|$.
\item If $g$ is convex and admits a minimizer (i.e., $\arg\min g\neq\emptyset$), then the (weak)
Pareto set equals
\[
\cP=(\arg\min g)\times \mathrm{conv}\{a_1,\dots,a_m\}.
\]
In particular, $\mathrm{conv}\{a_1,\dots,a_m\}$ is a non-singleton polytope. If $g$ is $\mu$-strongly convex, then $\arg\min g=\{x_V^\star\}$ is a singleton and
$\cP=\{x_V^\star\}\times \mathrm{conv}\{a_1,\dots,a_m\}$.
\item Regardless of convexity of $g$, the Pareto-stationary set of the lifted instance satisfies
\[
\cS:=\{x:\cG(x)=0\}=\{x_V\in V:\grad g(x_V)=0\}\times \mathrm{conv}\{a_1,\dots,a_m\}.
\]
\end{enumerate}
Consequently, any oracle lower bound on $\|\grad g(u^{(t)})\|$ for a scalar first-order method applied to $g$ transfers verbatim to a lower bound on $\cG(x^{(t)})$ for the lifted MOO instance.
\end{theorem}
\begin{proof}
Throughout, we use the orthogonal decomposition $\R^d=V\oplus W$ and write $x=x_V+x_W$ with $x_V\in V$ and $x_W\in W$. 
Differentiating \eqref{eq:lift-def} yields, for each $i$,
\begin{equation}\label{eq:lift-grad}
\nabla f_i(x)=\bigl(\nabla g(x_V),\ \gamma(x_W-a_i)\bigr).
\end{equation}

\medskip
\noindent\textbf{(i) Distinctness of objectives.}
Because the points $a_1,\ldots,a_m$ are affinely independent, in particular they are pairwise distinct. Fix $i\neq j$ and take any $x_V\in V$ with $x_W=a_i$. Then
\[
f_i(x_V+a_i)=g(x_V),
\qquad
f_j(x_V+a_i)=g(x_V)+\frac{\gamma}{2}\|a_i-a_j\|^2>g(x_V),
\]
so $f_i\not\equiv f_j$. Hence the objectives are distinct.

\medskip
\noindent\textbf{(ii) Smoothness and (strong) convexity.}
We first verify smoothness. Since $g$ has $L_g$-Lipschitz gradient on $V$, for all $u,v\in V$,
$\|\nabla g(u)-\nabla g(v)\|\le L_g\|u-v\|$. Using \eqref{eq:lift-grad}, we obtain for any $x,y\in\R^d$,
\begin{align*}
\|\nabla f_i(x)-\nabla f_i(y)\|^2
&=\|\nabla g(x_V)-\nabla g(y_V)\|^2+\gamma^2\|x_W-y_W\|^2 \\
&\le L_g^2\|x_V-y_V\|^2+\gamma^2\|x_W-y_W\|^2 \\
&\le \max\{L_g^2,\gamma^2\}\bigl(\|x_V-y_V\|^2+\|x_W-y_W\|^2\bigr) \\
&=\max\{L_g^2,\gamma^2\}\,\|x-y\|^2,
\end{align*}
where we used orthogonality of $V$ and $W$ in the last step. Taking square roots yields
$\|\nabla f_i(x)-\nabla f_i(y)\|\le \max\{L_g,\gamma\}\|x-y\|$, hence each $f_i$ is
$L$-smooth with $L=\max\{L_g,\gamma\}$.

If $g$ is convex, then $x\mapsto g(x_V)$ is convex on $\R^d$ (as a composition of the linear projection $x\mapsto x_V$ with a convex function), and $x\mapsto \frac{\gamma}{2}\|x_W-a_i\|^2$ is convex for $\gamma\ge 0$; therefore $f_i$ is convex.
If $g$ is $\mu$-strongly convex on $V$, then $x\mapsto g(x_V)$ is $\mu$-strongly convex along directions in $V$, while $x\mapsto \frac{\gamma}{2}\|x_W-a_i\|^2$ is $\gamma$-strongly convex along directions in $W$; consequently, for all $x,y\in\R^d$,
\[
f_i(y)\ge f_i(x)+\langle\nabla f_i(x),y-x\rangle+\frac{\min\{\mu,\gamma\}}{2}\|y-x\|^2,
\]
so $f_i$ is $\min\{\mu,\gamma\}$-strongly convex. In particular, if $\gamma\ge\mu$ then each $f_i$ is $\mu$-strongly convex.

\medskip
\noindent\textbf{(iii) Stationarity-gap formula and domination.}
For any $\lambda\in\Delta^m$, summing \eqref{eq:lift-grad} gives
\[
\sum_{i=1}^m \lambda_i \nabla f_i(x)
=\left(\nabla g(x_V),\ \gamma\Bigl(x_W-\sum_{i=1}^m \lambda_i a_i\Bigr)\right).
\]
By orthogonality of $V$ and $W$,
\[
\left\|\sum_{i=1}^m \lambda_i \nabla f_i(x)\right\|^2
=\|\nabla g(x_V)\|^2+\gamma^2\left\|x_W-\sum_{i=1}^m \lambda_i a_i\right\|^2.
\]
The first term is independent of $\lambda$, and $\{\sum_i \lambda_i a_i:\lambda\in\Delta^m\}=\mathrm{conv}\{a_1,\ldots,a_m\}$, hence
\begin{align*}
\cG(x)^2
&=\min_{\lambda\in\Delta^m}\left\|\sum_{i=1}^m \lambda_i \nabla f_i(x)\right\|^2 \\
&=\|\nabla g(x_V)\|^2+\gamma^2\,\mathrm{dist}\!\left(x_W,\mathrm{conv}\{a_1,\ldots,a_m\}\right)^2,
\end{align*}
which is \eqref{eq:gap-formula}. In particular, $\cG(x)\ge \|\nabla g(x_V)\|$.

\medskip
\noindent\textbf{(iv) Pareto set in the convex case.}
Assume $g$ is convex and admits a minimizer $x_V^\star\in\arg\min g$. By Proposition~\ref{prop:convex-iff}, weak Pareto optima coincide with minimizers of some scalarization $f_\lambda=\sum_{i=1}^m\lambda_i f_i$. For any $\lambda\in\Delta^m$,
\begin{align*}
f_\lambda(x_V,x_W)
&= g(x_V)+\frac{\gamma}{2}\sum_{i=1}^m \lambda_i\|x_W-a_i\|^2 \\
&= g(x_V)+\frac{\gamma}{2}\left\|x_W-\bar a(\lambda)\right\|^2+\mathrm{const}(\lambda),
\qquad \bar a(\lambda):=\sum_{i=1}^m \lambda_i a_i,
\end{align*}
where $\mathrm{const}(\lambda)$ is independent of $(x_V,x_W)$. Thus minimization decouples: the $V$-component is minimized by any $x_V\in\arg\min g$, and the quadratic in $x_W$ is uniquely minimized at $x_W=\bar a(\lambda)$. Hence
\[
\arg\min_{x\in\mathbb{R}^d} f_\lambda(x) = (\arg\min g)\times\{\bar a(\lambda)\}.
\]
As $\lambda$ ranges over $\Delta^m$, $\bar a(\lambda)$ ranges over $\mathrm{conv}\{a_1,\ldots,a_m\}$, yielding
\[
\cP=(\arg\min g)\times \mathrm{conv}\{a_1,\dots,a_m\}.
\]
Since $m\ge 2$ and the $a_i$ are affinely independent, $\mathrm{conv}\{a_i\}$ is a non-singleton
polytope. If $g$ is $\mu$-strongly convex, then $\arg\min g=\{x_V^\star\}$ is a singleton and the
same description reduces to $\cP=\{x_V^\star\}\times \mathrm{conv}\{a_1,\dots,a_m\}$.

\medskip
\noindent\textbf{(v) Pareto-stationary set for general $g$.}
By \eqref{eq:gap-formula}, $\cG(x)=0$ holds if and only if $\nabla g(x_V)=0$ and
$\mathrm{dist}(x_W,\mathrm{conv}\{a_i\})=0$, i.e., $x_W\in \mathrm{conv}\{a_1,\ldots,a_m\}$. Therefore,
\[
\cS=\{x:\cG(x)=0\}=\{x_V\in V:\nabla g(x_V)=0\}\times \mathrm{conv}\{a_1,\ldots,a_m\}.
\]

\medskip
\noindent\textbf{Transfer claim.}
The pointwise domination $\cG(x)\ge \|\nabla g(x_V)\|$ from (iii) implies that for any iterate sequence $(x^{(t)})_{t\ge0}$,
\[
\cG(x^{(t)})\ge \|\nabla g(x_V^{(t)})\|\qquad\text{for all }t.
\]
Hence any lower bound on $\|\nabla g(\cdot)\|$ along the projected iterates transfers immediately to a lower bound on $\cG(\cdot)$ along the lifted iterates.
\end{proof}

\subsection{Scalar hard instances: polynomial-degree and geometric hardness}\label{sec:scalar-hard}

We will use two types of scalar lower bounds:
(i) \emph{quadratic} lower bounds for span/oblivious methods that reduce to extremal polynomial problems; and
(ii) \emph{geometric} (zero-chain) lower bounds for finding stationary points under a bounded initial value gap, due to \cite{CarmonDuchiHinderSidford2020, CarmonDuchiHinderSidford2021}.

\subsubsection{Quadratic hard instances for span/oblivious methods}

We first construct scalar quadratic hard instances tailored to span/oblivious first-order methods.
For quadratics, these methods generate iterates that admit a polynomial/Krylov representation, which enables explicit finite-horizon lower bounds on $\min_{t\le T}\|\nabla g(x^{(t)})\|$ and separates the behavior of the different method classes. Lemma~\ref{lem:hard-g} below states the specific scalar quadratic instances and bounds that will serve as base ingredients for our multiobjective lower bounds via the lifting theorem.

\begin{lemma}[Hard scalar quadratic instances]\label{lem:hard-g}
Fix $T\in\mathbb{N}$ and $R>0$.
\begin{enumerate}[label=(\alph*),leftmargin=2em]
\item \textup{(Strongly convex case)} Fix $L\ge\mu>0$ and $\kappa:=L/\mu$. For any span first-order method making $T$ oracle calls on quadratics, there exist a dimension $d \ge T+1$,  an $L$-smooth $\mu$-strongly convex quadratic $g$ and an initialization with $\|x^{(0)}-x^\star\|=R$ such that the $T$-th iterate satisfies
\[
\|\nabla g(x^{(T)})\|\;\ge\;
\mu\,R\,\frac{2}{\rho^T+\rho^{-T}}
\;\ge\;
\mu\,R\left(\frac{\sqrt{\kappa}-1}{\sqrt{\kappa}+1}\right)^T,
\qquad
\rho:=\frac{\sqrt{\kappa}+1}{\sqrt{\kappa}-1}.
\]
\item \textup{(Convex, oblivious one-step)} Fix $L>0$. For any oblivious one-step gradient method run for $T$ steps with step sizes $\alpha_t\in[0,1/L]$, there exists an $L$-smooth convex quadratic $g$ and an initialization with $\|x^{(0)}-x^\star\|=R$, where $x^\star\in\arg\min g$ is chosen as a closest minimizer to $x^{(0)}$, such that
\[
\min_{0\le t\le T}\|\nabla g(x^{(t)})\|\;\ge\;\frac{LR}{4(T+1)}.
\]
\item \textup{(Convex, oblivious span: polynomial-degree)} Fix $L>0$. For any oblivious span first-order method run for $T$ steps, there exists an $L$-smooth convex quadratic $g$ and an initialization with $\|x^{(0)}-x^\star\|=R$, where $x^\star\in\arg\min g$ is chosen as a closest minimizer to $x^{(0)}$, such that
\[
\|\nabla g(x^{(T)})\|\;\ge\;\frac{LR}{2(T+1)^2}.
\]
\end{enumerate}
\end{lemma}

\begin{proof}
We prove each part separately.

\smallskip
\noindent\textbf{(a) Strongly convex case.}
Let $g(x)=\frac12x^\top Hx-b^\top x$ with $H\succ 0$. For any iterate sequence $\{x^{(t)}\}$ produced by a span first-order method on a quadratic, the error admits the polynomial representation (Lemma~\ref{lem:krylov}):
\begin{equation}\label{eq:poly-repr}
x^{(t)}-x^\star \;=\; p_t(H)\,\bigl(x^{(0)}-x^\star\bigr),
\qquad \deg p_t\le t,\quad p_t(0)=1.
\end{equation}
Set $e^{(0)}:=x^{(0)}-x^\star$ with $\|e^{(0)}\|=R$. Since $H\succeq \mu I$, we have for any vector $v$ that $\|Hv\|\ge \mu \|v\|$, hence
\begin{equation}\label{eq:grad-lb}
\bigl\|\nabla g(x^{(T)})\bigr\|
=\bigl\|H\,p_T(H)\,e^{(0)}\bigr\|
\;\ge\;\mu\,\bigl\|p_T(H)\,e^{(0)}\bigr\|.
\end{equation}
We now take worst case over the instance $(H,e^{(0)})$ with $\sigma(H)\subset[\mu,L]$ and $\|e^{(0)}\|=R$. For any fixed degree-$T$ polynomial $p_T$ with $p_T(0)=1$,
\[
\sup_{\substack{\sigma(H)\subset[\mu,L]\\ \|e^{(0)}\|=R}}\bigl\|p_T(H)\,e^{(0)}\bigr\|
\;=\; R\;\sup_{\zeta\in[\mu,L]} |p_T(\zeta)|,
\]
because for $H$ symmetric the operator norm $\|p_T(H)\|$ equals $\max_{\zeta\in\sigma(H)}|p_T(\zeta)|$, and we may align $e^{(0)}$ with an eigenvector where the maximum is attained.
Therefore, combining with \eqref{eq:grad-lb} and then minimizing over the (algorithm-induced) polynomial $p_T$ yields
\begin{equation}\label{eq:minimax}
\sup_{\substack{\sigma(H)\subset[\mu,L],\,\|e^{(0)}\|=R}}
\bigl\|\nabla g(x^{(T)})\bigr\|
\;\ge\; \mu R \cdot
\inf_{\substack{\deg p_T\le T\\ p_T(0)=1}}
\;\max_{\zeta\in[\mu,L]} |p_T(\zeta)|.
\end{equation}
The minimax problem on the right-hand side is the classical Chebyshev extremal problem on an interval with an off-interval normalization (at $\zeta=0$). The solution is detailed in Appendix~\ref{app:cheb-strong}. Let
\[
\xi \;=\; \frac{2\zeta-(L+\mu)}{L-\mu}\in[-1,1], 
\qquad \xi_0 \;=\; \frac{-\, (L+\mu)}{L-\mu} \;=\; -\,\frac{\kappa+1}{\kappa-1},
\]
and define $\tilde p_T(\xi):=p_T(\zeta(\xi))$. Then the extremal value is achieved by the scaled Chebyshev polynomial of the first kind,
\[
\tilde p_T^\star(\xi)\;=\;\frac{T_T(\xi)}{T_T(\xi_0)},\qquad
\min_{\deg\tilde p_T\le T,\;\tilde p_T(\xi_0)=1}\;\max_{\xi\in[-1,1]}|\tilde p_T(\xi)|
\;=\;\frac{1}{|T_T(\xi_0)|}.
\]
For $|x|>1$ one has $T_T(x)=\tfrac12\bigl(\rho^T+\rho^{-T}\bigr)$ with $\rho=x+\sqrt{x^2-1}$. Substituting $x=|\xi_0|=\frac{\kappa+1}{\kappa-1}$ gives
\[
|T_T(\xi_0)| \;=\; \frac{\rho^T+\rho^{-T}}{2}, 
\qquad \rho=\frac{\sqrt{\kappa}+1}{\sqrt{\kappa}-1}.
\]
Thus the extremal value of \eqref{eq:minimax} equals $\mu R \cdot \frac{2}{\rho^T+\rho^{-T}}$, proving
\[
\sup_{\text{instances}} \bigl\|\nabla g(x^{(T)})\bigr\|
\;\ge\; \mu R \cdot \frac{2}{\rho^T+\rho^{-T}}
\;\ge\; \mu R \left(\frac{\sqrt{\kappa}-1}{\sqrt{\kappa}+1}\right)^T,
\]
where the last inequality follows from $\frac{2}{a+a^{-1}}\ge a^{-1}$ for $a\ge 1$.

\smallskip
\noindent\textbf{(b) Convex oblivious one-step case.}
This is precisely Corollary~\ref{cor:oblivious-gradient-lb} in Appendix~\ref{app:prod-poly-tight} (with $t=T$).

\smallskip
\noindent\textbf{(c) Convex oblivious span polynomial-degree case.}
Fix an arbitrary \emph{oblivious span} method and consider any $L$-smooth convex quadratic
$g(x)=\tfrac12 x^\top Hx-b^\top x$ with $\sigma(H)\subset[0,L]$.
By Lemma~\ref{lem:krylov} (oblivious case), there exists a polynomial $p_T$ with $\deg p_T\le T$ and $p_T(0)=1$
that depends only on the method (and on $T$), such that
\[
x^{(T)}-x^\star = p_T(H)\,(x^{(0)}-x^\star),
\]
and hence
\[
\nabla g(x^{(T)}) = H\,p_T(H)\,(x^{(0)}-x^\star).
\]
Let $\|x^{(0)}-x^\star\|=R$. For any fixed $\zeta\in[0,L]$, choose $H$ diagonal with a single eigenvalue $\zeta$
(and all remaining eigenvalues equal to $0$) and align $x^{(0)}-x^\star$ with the corresponding eigenvector. Then
\[
\|\nabla g(x^{(T)})\| \;=\; R\,|\zeta\,p_T(\zeta)|.
\]
Since $p_T$ is fixed independently of $\zeta$ (obliviousness), taking the worst case over $\zeta\in[0,L]$ yields
\[
\sup_{\substack{g\ \text{$L$-smooth convex quadratic}\\ \|x^{(0)}-x^\star\|=R}}
\|\nabla g(x^{(T)})\|
\;\ge\;
R\,\max_{\zeta\in[0,L]}|\zeta p_T(\zeta)|.
\]
Finally, taking the infimum over oblivious span methods is lower bounded by the infimum over all admissible polynomials,
and Appendix~\ref{app:markov} gives
\[
\inf_{\substack{\deg p_T\le T\\ p_T(0)=1}}\ \max_{\zeta\in[0,L]} |\zeta p_T(\zeta)|
\;\ge\; \frac{L}{2(T+1)^2}.
\]
This yields $\|\nabla g(x^{(T)})\|\ge \frac{LR}{2(T+1)^2}$ as claimed.
\end{proof}

\subsubsection{Geometric stationary-point lower bounds under value-gap bounds}

We next state the geometric scalar stationary-point lower bounds that drive our new convex and nonconvex MOO results.

\begin{lemma}[Convex $L$-smooth stationary-point lower bound under bounded value gap {\cite[Theorem~1]{CarmonDuchiHinderSidford2021}}]\label{lem:carmon-convex}
Fix the initial point $x^{(0)}=0$.
A \emph{deterministic first-order method} $\mathcal A$ is any procedure that, for
$t=0,1,2,\dots$, adaptively chooses query points
\[
x^{(t+1)}=\Phi_t\!\left(\{(g(x^{(s)}),\nabla g(x^{(s)}))\}_{s=0}^{t}\right),
\]
for some deterministic mappings $\Phi_t$, where the oracle returns $(g(x),\nabla g(x))$
at each query.

Fix $L>0$ and $\Delta>0$. For any deterministic first-order method $\mathcal A$
and any $T\in\mathbb N$, there exist a dimension $d\ge 2T$ and a convex function
$g:\mathbb R^d\to\mathbb R$ with $L$-Lipschitz gradient such that
\[
g(x^{(0)})-\inf_{x\in\mathbb R^d} g(x)\le \Delta,
\]
and the iterates $\{x^{(t)}\}_{t=0}^{T}$ generated by $\mathcal A$ on $g$ satisfy
\[
\min_{0\le t\le T}\|\nabla g(x^{(t)})\|
\;\ge\;
\frac{\sqrt{L\Delta}}{4(T+1)}.
\]
Equivalently, any deterministic first-order method requires
$T=\Omega\!\left(\sqrt{L\Delta}/\varepsilon\right)$ oracle calls in the worst case
to guarantee an output $x$ with $\|\nabla g(x)\|\le \varepsilon$ over this function class.
\end{lemma}

\begin{proof}
Carmon, Duchi, Hinder, and Sidford \cite[Theorem~1]{CarmonDuchiHinderSidford2021}
prove the following: for any $\varepsilon>0$ and any deterministic first-order method $\mathcal A$,
there exists a convex function $f$ with $L$-Lipschitz gradient and $f(x^{(0)})-\inf_x f(x)\le \Delta$
such that $\mathcal A$ must make at least $\frac{\sqrt{L\Delta}}{4\varepsilon}$ first-order oracle calls
to guarantee finding a point $x$ with $\|\nabla f(x)\|\le \varepsilon$.
(Their statement is given for an orthogonally invariant function class $K_1(\Delta,L)$,
and the reduction from zero-respecting to general deterministic methods is handled via a
resisting-oracle / rotation argument, which increases the dimension to at most $2T$.)

Now fix $T\in\mathbb N$ and set $\varepsilon := \frac{\sqrt{L\Delta}}{4(T+1)}$.
Then $\frac{\sqrt{L\Delta}}{4\varepsilon}=T+1$, so any method using only $T$ oracle calls
cannot guarantee $\|\nabla g(x)\|\le \varepsilon$ on the above convex $L$-smooth class.
Therefore, for the corresponding worst-case instance $g$, all iterates up to time $T$ satisfy
$\|\nabla g(x^{(t)})\|>\varepsilon$, and hence
\[
\min_{0\le t\le T}\|\nabla g(x^{(t)})\|\ge \varepsilon
=\frac{\sqrt{L\Delta}}{4(T+1)}.
\]
Rearranging yields the equivalent complexity form $T=\Omega(\sqrt{L\Delta}/\varepsilon)$.
\end{proof}

\begin{lemma}[Nonconvex $L$-smooth stationary-point lower bound (Lipschitz gradient only) {\cite[Theorem~1]{CarmonDuchiHinderSidford2020}}]\label{lem:nonconvex-scalar}
Fix the initial point $x^{(0)}=0$.
A \emph{deterministic first-order method} $\mathcal A$ is any procedure that, for $t=0,1,2,\dots$,
adaptively chooses query points
\[
x^{(t+1)} = \Phi_t\!\left(\{(g(x^{(s)}),\nabla g(x^{(s)}))\}_{s=0}^t\right),
\]
for some deterministic mappings $\Phi_t$, where the oracle returns the pair $(g(x),\nabla g(x))$ at each query.

There exists a universal constant $c>0$ such that for any $L>0$, $\Delta>0$, any deterministic
first-order method $\mathcal A$, and any $T\in\mathbb N$, there exist a dimension $d\ge 2T+1$
and an $L$-smooth (not necessarily convex) function $g:\mathbb R^d\to\mathbb R$ satisfying
\[
g(x^{(0)})-\inf_{x\in\mathbb R^d} g(x)\le \Delta,
\]
such that the iterates $\{x^{(t)}\}_{t=0}^T$ generated by $\mathcal A$ on $g$ obey
\[
\min_{0\le t\le T}\|\nabla g(x^{(t)})\|\;\ge\; c\,\sqrt{\frac{L\Delta}{T+1}}.
\]
Equivalently, any deterministic first-order method requires $T=\Omega(L\Delta/\varepsilon^2)$
oracle calls in the worst case to guarantee an output $x$ with $\|\nabla g(x)\|\le \varepsilon$.
\end{lemma}

\begin{proof}
Carmon, Duchi, Hinder, and Sidford prove (see \cite[Theorem~1]{CarmonDuchiHinderSidford2020}
specialized to $p=1$) that there exists a universal constant $C>0$ such that for every $\varepsilon>0$,
every deterministic first-order method $\mathcal A$ requires at least
\[
T_\varepsilon \;\ge\; C\,\frac{L\Delta}{\varepsilon^2}
\]
oracle calls in the worst case over the class of $L$-smooth functions with initial gap at most $\Delta$,
i.e., there exists an $L$-smooth (possibly nonconvex) $g$ with $g(x^{(0)})-\inf g\le \Delta$ such that
$\|\nabla g(x^{(t)})\|>\varepsilon$ for all $t<T_\varepsilon$.

Now fix $T\in\mathbb N$ and set $\varepsilon := \sqrt{\frac{CL\Delta}{T+1}}$.
Then $T+1 < C\frac{L\Delta}{\varepsilon^2}$, hence $T < T_\varepsilon$ for the corresponding worst-case
instance $g$, which implies $\|\nabla g(x^{(t)})\|>\varepsilon$ for all $t\le T$ and therefore
\[
\min_{0\le t\le T}\|\nabla g(x^{(t)})\| \;\ge\; \varepsilon
\;=\; \sqrt{\frac{CL\Delta}{T+1}}.
\]
Renaming $c:=\sqrt{C}$ yields the stated bound.
The equivalence to $T=\Omega(L\Delta/\varepsilon^2)$ follows by rearranging the inequality.
\end{proof}

\subsection{Strongly Convex MOO: A Tight Linear Rate}
We now state the corresponding multiobjective lower bound in the strongly convex regime.

\begin{theorem}[Strongly Convex Lower Bound]\label{thm:strong-last-iterate}
Fix $L\ge \mu>0$, $T\in\mathbb{N}$, $\kappa:=L/\mu$, and $R>0$. Consider any span first-order method that makes $T$ first-order oracle calls when applied to a quadratic instance. Then there exist integers $m \ge 2$ and $d \ge (T+1) + (m-1)$, an MOO instance $(f_1,\dots,f_m):\mathbb{R}^{d}\to\mathbb{R}^m$
with $m$ objectives each $L$-smooth and $\mu$-strongly convex, and an initialization $x^{(0)}$ with $\mathrm{dist}(x^{(0)},\cP)=R$ such that the $T$-th iterate produced by the method satisfies
\begin{align*}
\cG\bigl(x^{(T)}\bigr)
  &\;\ge\; \mu\,R\;\frac{2}{\rho^T+\rho^{-T}} \\
  &\;\ge\; \mu\,R\left(\frac{\sqrt{\kappa}-1}{\sqrt{\kappa}+1}\right)^T,
  \qquad \rho := \frac{\sqrt{\kappa}+1}{\sqrt{\kappa}-1}.
\end{align*}
In particular, any such method has worst-case iteration complexity 
\[
\Omega\!\bigl(\sqrt{\kappa}\log(1/\varepsilon)\bigr)
\]
to reduce the Pareto stationarity gap below $\varepsilon$.
\end{theorem}

\begin{proof}
By Lemma~\ref{lem:hard-g}(a), there exists a dimension $\dim(V)\ge T+1$ and an $L$-smooth $\mu$-strongly convex quadratic $g$ on a subspace $V$ that forces the stated lower bound on $\|\nabla g(u^{(T)})\|$ for any span method (with $\|u^{(0)}-u^\star\|=R$). Apply the lifting construction of Theorem~\ref{thm:lifting} with $\gamma=\mu$ to obtain an MOO instance $(f_1,\dots,f_m)$ with $\mathrm{dist}(x^{(0)},\cP)=R$ and
\[
\cG(x^{(T)})\;\ge\;\|\nabla g(x_V^{(T)})\|.
\]
Since the $V$-component of the iterate sequence produced by a span method on the lifted instance satisfies the span property with respect to $g$, the scalar lower bound transfers, proving the claim. The iteration-complexity statement follows by solving $\mu R\bigl((\sqrt{\kappa}-1)/(\sqrt{\kappa}+1)\bigr)^T\le \varepsilon$ for $T$.
\end{proof}

\subsection{Convex MOO: Oblivious vs. Adaptive Methods}

\subsubsection{Oblivious One-Step Gradient Methods: An $\Omega(1/T)$ Min-Iterate Bound}

For the more restricted class of oblivious one-step gradient methods, we establish a lower bound of order $\Omega(1/T)$. This bound applies to the minimum Pareto stationarity gap achieved over the first $T$ iterates and demonstrates a fundamental performance limit for non-adaptive algorithms that do not incorporate momentum or history-dependent updates.

\begin{theorem}[Oblivious Convex Lower Bound]\label{thm:oblivious-lower}
Fix $L,R>0$ and $T\in\mathbb{N}$. For any oblivious one-step gradient method run for $T$ steps, there exists an MOO instance with $L$-smooth convex objectives and $\mathrm{dist}(x^{(0)}, \cP) = R$ such that
\[
\min_{0\le t\le T} \cG(x^{(t)}) \;\ge\; \frac{LR}{4(T+1)}.
\]
\end{theorem}
\begin{proof}
By Lemma~\ref{lem:hard-g}(b), for the given stepsize schedule there exists an
$L$-smooth convex quadratic $g$ and an initialization $u^{(0)}$ such that
\[
\min_{0\le t\le T}\|\nabla g(u^{(t)})\|\;\ge\;\frac{LR}{4(T+1)}
\qquad\text{and}\qquad
\mathrm{dist}\!\bigl(u^{(0)},\arg\min g\bigr)=R.
\]
Choose $u^\star\in\arg\min g$ to be a closest minimizer to $u^{(0)}$, so that
$\|u^{(0)}-u^\star\|=\mathrm{dist}(u^{(0)},\arg\min g)=R$.

Apply Theorem~\ref{thm:lifting} with $\gamma=L$ to lift $g$ into a non-degenerate
MOO instance $(f_1,\dots,f_m)$.
By Theorem~\ref{thm:lifting}(iv), the (weak) Pareto set satisfies
\[
\cP=(\arg\min g)\times \mathrm{conv}\{a_1,\dots,a_m\}.
\]
Pick $w^{(0)}\in\mathrm{conv}\{a_1,\dots,a_m\}$ (e.g., $w^{(0)}=a_1$) and set
$x^{(0)}=(u^{(0)},w^{(0)})\in V\oplus W$. Then
\[
\mathrm{dist}\!\bigl(x^{(0)},\cP\bigr)=\mathrm{dist}\!\bigl(u^{(0)},\arg\min g\bigr)=R.
\]

Finally, Theorem~\ref{thm:lifting}(iii) gives $\cG(x)\ge \|\nabla g(x_V)\|$ for all $x$.
Hence,
\[
\min_{0\le t\le T}\cG(x^{(t)})
\;\ge\;
\min_{0\le t\le T}\|\nabla g(x_V^{(t)})\|
\;\ge\;
\frac{LR}{4(T+1)},
\]
as claimed.
\end{proof}

\subsubsection{Oblivious span methods: polynomial-degree vs.\ geometric lower bounds}

For oblivious span methods, the error polynomial $p_t$ is no longer restricted to have the product form of the one-step class, but it remains instance-independent. This leads to a universal $1/T^2$ \emph{last-iterate} lower bound based on Markov's inequality for polynomials. This bound is independent of geometric ``chain'' constructions and isolates a limitation coming solely from polynomial degree constraints within the oblivious span
model.

For fully adaptive first-order methods (allowed to use the entire oracle transcript), we will also show a stronger geometric $\Omega(1/T)$ \emph{min-iterate} lower bound under a bounded scalarization gap by lifting the
stationary-point lower bounds of Carmon--Duchi--Hinder--Sidford \cite{CarmonDuchiHinderSidford2021}.

\begin{theorem}[Polynomial-degree lower bound for oblivious span methods (last iterate)]\label{thm:oblivious-span-polydegree-lower}
Fix $L,R>0$ and $T\in\mathbb{N}$. For any oblivious span first-order method run for $T$ steps, there exists an MOO instance with $L$-smooth convex objectives and $\mathrm{dist}(x^{(0)}, \cP) = R$ such that the last iterate obeys
\[
\cG(x^{(T)}) \;\ge\; \frac{LR}{2\,(T+1)^2}.
\]
\end{theorem}
\begin{proof}
By Lemma~\ref{lem:hard-g}(c), there exists an $L$-smooth convex quadratic $g$ and an
initialization $u^{(0)}$ such that for any oblivious span method run for $T$ steps,
\[
\|\nabla g(u^{(T)})\|\;\ge\;\frac{LR}{2\,(T+1)^2}
\qquad\text{and}\qquad
\mathrm{dist}\!\bigl(u^{(0)},\arg\min g\bigr)=R.
\]
Choose $u^\star\in\arg\min g$ to be a closest minimizer to $u^{(0)}$, so that
$\|u^{(0)}-u^\star\|=\mathrm{dist}(u^{(0)},\arg\min g)=R$.

Apply Theorem~\ref{thm:lifting} with $\gamma=L$ to lift $g$ into a non-degenerate
MOO instance $(f_1,\dots,f_m)$. By Theorem~\ref{thm:lifting}(iv),
\[
\cP=(\arg\min g)\times \mathrm{conv}\{a_1,\dots,a_m\}.
\]
Pick $w^{(0)}\in\mathrm{conv}\{a_1,\dots,a_m\}$ (e.g., $w^{(0)}=a_1$) and set
$x^{(0)}=(u^{(0)},w^{(0)})\in V\oplus W$. Then
\[
\mathrm{dist}\!\bigl(x^{(0)},\cP\bigr)=\mathrm{dist}\!\bigl(u^{(0)},\arg\min g\bigr)=R.
\]

Finally, Theorem~\ref{thm:lifting}(iii) gives $\cG(x)\ge \|\nabla g(x_V)\|$ for all $x$.
Therefore,
\[
\cG(x^{(T)})
\;\ge\;
\|\nabla g(x_V^{(T)})\|
\;\ge\;
\frac{LR}{2\,(T+1)^2},
\]
as claimed.
\end{proof}

\begin{remark}{Iterate selection.} Theorem~\ref{thm:oblivious-lower} is a \emph{min-iterate} lower bound, whereas Theorem~\ref{thm:oblivious-span-polydegree-lower} is a \emph{last-iterate} bound; the two statements are therefore not directly comparable a priori.
\end{remark}

\begin{remark}[The polynomial-degree bound is loose]\label{rem:loose}
Theorem~\ref{thm:oblivious-span-polydegree-lower} is a \emph{last-iterate} bound derived from polynomial extremal inequalities and does not exploit geometric hard instances (e.g., zero-chains). As we show next, geometric stationary-point lower bounds imply a stronger \emph{min-iterate} $\Omega(1/T)$ lower bound for adaptive methods in smooth convex MOO under a bounded scalarization gap.
\end{remark}

We now lift the convex stationary-point lower bound of Carmon--Duchi--Hinder--Sidford \cite{CarmonDuchiHinderSidford2021}. The appropriate ``value gap'' assumption in MOO is expressed through a scalarization $f_\lambda$.

Our lifting construction (Theorem~\ref{thm:lifting}) produces a multiobjective instance $f_1,\dots,f_m$ from a scalar function $g$ such that the Pareto stationarity gap dominates the scalar gradient norm, i.e., $\cG(x)\ge \|\nabla g(x_V)\|$ for every $x=(x_V,x_W)$. For lower bounds proved for restricted algorithm classes (e.g., oblivious one-step or (oblivious) span methods), this domination property is sufficient: the iterates must remain in prescribed linear spans / polynomial images of past first-order information, so the $V$-component of the iterates on the lifted instance inherits the same restriction, and the corresponding scalar lower bounds transfer directly via $\cG(x)\ge \|\nabla g(x_V)\|$.

In contrast, the lower bounds under a bounded initial value gap $\Delta$ (such as those of \cite{CarmonDuchiHinderSidford2021}) are minimax statements over \emph{all} deterministic first-order methods. Here an additional reduction is required because a multiobjective first-order oracle returns strictly richer information than a scalar oracle: at each query $x$ it provides the full collection $\{(f_i(x),\nabla f_i(x))\}_{i=1}^m$, whereas in the scalar setting one observes only $(g(x_V),\nabla g(x_V))$. To legitimately invoke such scalar minimax lower bounds, we therefore use a \emph{simulation} (transcript) argument: for the lifted family, any deterministic multiobjective algorithm can be simulated by a deterministic scalar first-order method on $g$ with the same number of oracle calls (since the multiobjective replies can be reconstructed from $(g(x_V),\nabla g(x_V))$ and the known quadratic $W$-terms). Consequently, if a multiobjective method were able to beat the scalar value-gap lower bound on the lifted instances, the simulated scalar method would also beat it, yielding a contradiction.

\begin{theorem}[Geometric convex lower bound for adaptive methods under a bounded scalarization gap]\label{thm:convex-geometric}
Fix $L>0$, $\Delta>0$, $T\in\mathbb{N}$, and $m\ge2$. Consider the standard \emph{vector first-order oracle} for MOO that, at any query $x\in\R^d$, returns the transcript item
\[
\mathcal{O}_F(x)\;:=\;\bigl( (f_i(x),\nabla f_i(x))\bigr)_{i=1}^m.
\]
Let $\mathsf{Alg}$ be any deterministic first-order method that makes $T$ oracle calls to $\mathcal{O}_F$ and outputs iterates
$x^{(0)},x^{(1)},\ldots,x^{(T)}$.
Then there exists an $m$-objective MOO instance $F=(f_1,\ldots,f_m)$ with the following properties:
\begin{enumerate}[label=(\roman*),leftmargin=2em]
\item Each $f_i$ is convex and $L$-smooth, the objectives are distinct, and the Pareto set $\cP$ is non-singleton (in fact, it contains a non-singleton polytope subset).
\item \emph{(Bounded scalarization gap.)} For the scalarization $\lambda=e_1$,
\[
\Delta_{e_1}(x^{(0)}) \;:=\; f_{e_1}(x^{(0)})-\inf_x f_{e_1}(x)
\;=\; f_1(x^{(0)})-\inf_x f_1(x)\;\le\;\Delta.
\]
\item \emph{(Min-iterate lower bound.)} The iterates satisfy
\[
\min_{0\le t\le T}\cG(x^{(t)})\;\ge\;\frac{\sqrt{L\Delta}}{4(T+1)}.
\]
\end{enumerate}
Equivalently, any deterministic first-order method needs $T=\Omega(\sqrt{L\Delta}/\varepsilon)$ oracle calls to guarantee $\min_{t\le T}\cG(x^{(t)})\le\varepsilon$ on this class of instances.
\end{theorem}

\begin{proof}
\noindent\textbf{Oracle model as deterministic transcript maps.}
A deterministic first-order method $\mathsf{Alg}$ interacting with $\mathcal{O}_F$ can be formalized as a sequence of deterministic maps
\[
\Phi_t:\Big(\R^d\times (\R\times\R^d)^m\Big)^{t}\to\R^d,\qquad t\ge 0,
\]
such that, for $t=0,1,\ldots,T-1$,
\[
x^{(t+1)}=\Phi_t\Big(x^{(0)},\mathcal{O}_F(x^{(0)}),\ldots,x^{(t)},\mathcal{O}_F(x^{(t)})\Big).
\]
This is the standard deterministic first-order oracle model; $\mathsf{Alg}$ may use all past queried points and oracle replies.

\medskip
\noindent\textbf{Step 1: Choose a scalar hard instance from Carmon et al.}
Apply Lemma~\ref{lem:carmon-convex} to the deterministic first-order method that will be constructed below (it will make $T$ scalar first-order oracle calls).
Thus there exists a convex $L$-smooth function $g:V\to\R$ (with $\dim(V)\ge 2T$) and an initialization $u^{(0)}\in V$ satisfying
\begin{equation}\label{eq:carmon-gap}
g(u^{(0)})-\inf_{u\in V} g(u)\le \Delta
\end{equation}
and such that the iterates produced by that scalar method obey
\begin{equation}\label{eq:carmon-lb}
\min_{0\le t\le T}\|\nabla g(u^{(t)})\|
\;\ge\;\frac{\sqrt{L\Delta}}{4(T+1)}.
\end{equation}

\medskip
\noindent\textbf{Step 2: Lift $g$ to a non-degenerate MOO instance.}
Let $W$ be an orthogonal complement of $V$ with $\dim(W)\ge m-1$ and pick affinely independent points $a_1,\ldots,a_m\in W$.
Define $F=(f_1,\ldots,f_m)$ on $\R^d=V\oplus W$ by the lifting construction (Theorem~\ref{thm:lifting}) with $\gamma=L$:
\[
f_i(u,w)\;:=\; g(u)+\frac{L}{2}\|w-a_i\|^2,\qquad i=1,\ldots,m.
\]
By Theorem~\ref{thm:lifting}(i)--(iv), each $f_i$ is convex and $L$-smooth, the objectives are distinct, and the Pareto set is
\[
\cP=(\arg\min g)\times\mathrm{conv}\{a_1,\ldots,a_m\},
\]
hence non-singleton. In particular, for any $u^\star\in\arg\min g$, the slice
$\{u^\star\}\times \mathrm{conv}\{a_1,\dots,a_m\}\subseteq \cP$ is a non-singleton polytope.

We choose the initialization
\[
x^{(0)}:=(u^{(0)},a_1).
\]

\medskip
\noindent\textbf{Step 3: Verify the bounded scalarization gap $\Delta_{e_1}(x^{(0)})\le\Delta$.}
For $\lambda=e_1$, $f_{e_1}=f_1$ and
\[
f_1(u,w)=g(u)+\frac{L}{2}\|w-a_1\|^2.
\]
Hence $\inf_{(u,w)} f_1(u,w)=\inf_u g(u)$ (attained at $(u^\star,a_1)$), and therefore
\[
\Delta_{e_1}(x^{(0)})
=f_1(u^{(0)},a_1)-\inf_{(u,w)} f_1(u,w)
=g(u^{(0)})-\inf_u g(u)
\le \Delta,
\]
by \eqref{eq:carmon-gap}. This proves item (ii).
\emph{Importantly, this does not make the MOO instance degenerate: for $i\neq1$, the objectives differ because $a_i\neq a_1$. The role of $e_1$ is only to express a scalar \emph{value-gap} condition needed to invoke Lemma~\ref{lem:carmon-convex}.}

\medskip
\noindent\textbf{Step 4: A transcript-level reduction from MOO to the scalar oracle for $g$.}
Let $\mathcal{O}_g$ be the scalar first-order oracle for $g$, returning $\mathcal{O}_g(u):=(g(u),\nabla g(u))$.
We now define a deterministic scalar first-order method $\widetilde{\mathsf{Alg}}$ interacting with $\mathcal{O}_g$ that \emph{simulates} the transcript of $\mathsf{Alg}$ interacting with $\mathcal{O}_F$ on the lifted instance.

Inductively suppose $\widetilde{\mathsf{Alg}}$ has constructed points $(u^{(0)},w^{(0)}),\ldots,(u^{(t)},w^{(t)})$ equal to the iterates of $\mathsf{Alg}$ on $F$.
It queries $\mathcal{O}_g$ at $u^{(t)}$ and receives $(g(u^{(t)}),\nabla g(u^{(t)}))$.
Using the known lift parameters $L$ and $(a_i)_{i=1}^m$, it reconstructs the full vector-oracle reply at $x^{(t)}=(u^{(t)},w^{(t)})$ as
\[
\mathcal{O}_F(x^{(t)})=
\Bigl(g(u^{(t)})+\tfrac{L}{2}\|w^{(t)}-a_i\|^2,\ (\nabla g(u^{(t)}),\,L(w^{(t)}-a_i))\Bigr)_{i=1}^m.
\]
Since $\mathsf{Alg}$ is deterministic and its next iterate is given by the deterministic map $\Phi_t$ applied to the transcript,
$\widetilde{\mathsf{Alg}}$ can compute
\[
x^{(t+1)}=\Phi_t\Big(x^{(0)},\mathcal{O}_F(x^{(0)}),\ldots,x^{(t)},\mathcal{O}_F(x^{(t)})\Big)
\]
and set $(u^{(t+1)},w^{(t+1)})=x^{(t+1)}$.
Thus the sequence $(u^{(t)})_{t=0}^T$ generated by $\widetilde{\mathsf{Alg}}$ is exactly the $V$-projection of the iterates generated by $\mathsf{Alg}$ on the lifted instance.
Moreover, $\widetilde{\mathsf{Alg}}$ makes exactly one scalar oracle call per iteration, hence $T$ calls total.

Therefore, the Carmon lower bound \eqref{eq:carmon-lb} applies to the simulated sequence $(u^{(t)})$.

\medskip
\noindent\textbf{Step 5: Transfer to the Pareto stationarity gap.}
By Theorem~\ref{thm:lifting}(iii), for all $t$,
\[
\cG(x^{(t)})\;\ge\;\|\nabla g(u^{(t)})\|.
\]
Taking the minimum over $0\le t\le T$ and using \eqref{eq:carmon-lb} gives
\[
\min_{0\le t\le T}\cG(x^{(t)})
\;\ge\;\min_{0\le t\le T}\|\nabla g(u^{(t)})\|
\;\ge\;\frac{\sqrt{L\Delta}}{4(T+1)}.
\]
This proves item (iii) and completes the proof.
\end{proof}

Theorem~\ref{thm:convex-geometric} is naturally stated in terms of a scalarization value gap $\Delta$. On our lifted family one can calibrate $\Delta$ and the Pareto distance scale $R$ simultaneously, yielding the $LR/T$ form.

\begin{corollary}[From $\Delta$ to $R$ on the lifted family]\label{cor:LRoverT}
Fix $L>0$, $R>0$, $T\in\mathbb{N}$, and $m\ge2$. There exists an MOO instance with $m$ distinct convex $L$-smooth objectives and an initialization $x^{(0)}$ such that:
\begin{enumerate}[label=(\roman*),leftmargin=2em]
\item $\mathrm{dist}(x^{(0)},\cP)=R$ and $\cP$ is non-singleton (in particular, it contains a non-singleton polytope subset);
\item for $\lambda=e_1$ the scalarization gap satisfies $\Delta_{e_1}(x^{(0)})\asymp L R^2$ (up to universal constants);
\item for any deterministic first-order method run for $T$ steps,
\[
\min_{0\le t\le T}\cG(x^{(t)})\;\ge\; c_0\,\frac{LR}{T+1}
\]
for a universal constant $c_0>0$.
\end{enumerate}
\end{corollary}

\begin{proof}
Let $(g_0,u_0^{(0)})$ be the base hard instance from Theorem~\ref{thm:convex-geometric} for parameters $(L_0,\Delta_0)$,
and denote $u_0^\star\in\arg\min g_0$ and $R_0:=\|u_0^{(0)}-u_0^\star\|>0$.
For $a,s>0$, define the scaled function $g(u):=a\, g_0(su)$ and initialization $u^{(0)}:=u_0^{(0)}/s$.
Then $\nabla g$ is $L=as^2L_0$-Lipschitz and
\[
g(u^{(0)})-\inf g = a\bigl(g_0(u_0^{(0)})-\inf g_0\bigr)=a\Delta_0=: \Delta.
\]
Moreover, since $\arg\min g=\{u_0^\star/s: u_0^\star\in\arg\min g_0\}$, we have
\[
\mathrm{dist}(u^{(0)},\arg\min g)=\frac{1}{s}\,\mathrm{dist}(u_0^{(0)},\arg\min g_0)=\frac{R_0}{s}.
\]
Choose $s:=R_0/R$ and $a:=L/(s^2L_0)$. Then $\nabla g$ is $L$-Lipschitz and
\[
\Delta=a\Delta_0=\frac{L\Delta_0}{s^2L_0}=\frac{\Delta_0}{L_0R_0^2}\,LR^2 \asymp LR^2,
\]
where the implied constants depend only on the fixed base instance.

Now lift $g$ using Theorem~\ref{thm:lifting} with $m\ge2$ distinct points $a_1,\dots,a_m$ in $W$ and parameter $\gamma=L$,
and set the initialization $x^{(0)}:=(u^{(0)},a_1)$.
By Theorem~\ref{thm:lifting}, the Pareto set is $\mathcal P=\arg\min g\times\mathrm{conv}\{a_i\}$, hence it is non-singleton,
and since $(u^\star,a_1)\in\mathcal P$ for any $u^\star\in\arg\min g$ we have
\[
\mathrm{dist}(x^{(0)},\mathcal P)=\mathrm{dist}(u^{(0)},\arg\min g)=R.
\]
Also, for $\lambda=e_1$ we have $f_\lambda=f_1(u,w)=g(u)+\frac{L}{2}\|w-a_1\|^2$, so
$\Delta_{e_1}(x^{(0)})=g(u^{(0)})-\inf g=\Delta\asymp LR^2$.

Finally, applying Theorem~\ref{thm:convex-geometric} yields
\[
\min_{0\le t\le T}\mathcal G(x^{(t)})\ge \frac{\sqrt{L\Delta}}{4(T+1)} \asymp \frac{LR}{T+1},
\]
which proves the claim for a universal constant $c_0>0$.
\end{proof}

\subsection{Nonconvex MOO with Lipschitz gradients}\label{sec:nonconvex}

We now use the convexity-free nature of Theorem~\ref{thm:lifting} to transfer classical nonconvex stationary-point lower bounds to Pareto stationarity. We emphasize that in nonconvex MOO, $\cG(x)=0$ is a necessary stationarity condition (Lemma~\ref{lem:equiv-stationarity}) but generally not sufficient for Pareto optimality.

\begin{theorem}[Nonconvex MOO lower bound under $L$-Lipschitz gradients]\label{thm:nonconvex}
There exists a universal constant $c>0$ such that for any $L>0$, $\Delta>0$, $T\in\mathbb{N}$, and $m\ge2$, the following holds. For any deterministic first-order method making $T$ vector-oracle calls, there exists an MOO instance with $m$ distinct objectives, each $L$-smooth (not necessarily convex), and an initialization $x^{(0)}$ such that:
\begin{enumerate}[label=(\roman*),leftmargin=2em]
\item For the scalarization $\lambda=e_1$, $\Delta_{e_1}(x^{(0)})\le \Delta$.
\item The iterates satisfy
\[
\min_{0\le t\le T}\cG(x^{(t)})\;\ge\; c\,\sqrt{\frac{L\Delta}{T+1}}.
\]
\end{enumerate}
Equivalently, any deterministic first-order method requires $T=\Omega(L\Delta/\varepsilon^2)$ oracle calls to guarantee $\min_{t\le T}\cG(x^{(t)})\le\varepsilon$.
\end{theorem}

\begin{proof}
Apply Lemma~\ref{lem:nonconvex-scalar} to obtain an $L$-smooth (possibly nonconvex) scalar function $g:V\to\R$ with $g(u^{(0)})-\inf g\le\Delta$ but
\[
\min_{0\le t\le T}\|\grad g(u^{(t)})\|\ge c\sqrt{\frac{L\Delta}{T+1}}
\]
for any deterministic first-order method run for $T$ steps.

Lift $g$ to an MOO instance via Theorem~\ref{thm:lifting} with $\gamma=L$ and initialization $x^{(0)}=(u^{(0)},w^{(0)})$ where $w^{(0)}\in\mathrm{conv}\{a_i\}$. As in the convex proof, $\Delta_{e_1}(x^{(0)})=g(u^{(0)})-\inf g\le \Delta$. Moreover, by \eqref{eq:gap-formula},
\[
\cG(x^{(t)})\ge \|\grad g(u^{(t)})\|\quad\forall t.
\]
Finally, the same simulation argument as in Theorem~\ref{thm:convex-geometric} shows that the $V$-coordinates of any MOO first-order method on the lifted instance correspond to a valid scalar first-order method on $g$, so the scalar lower bound transfers:
\[
\min_{0\le t\le T}\cG(x^{(t)})\ge \min_{0\le t\le T}\|\grad g(u^{(t)})\|\ge c\sqrt{\frac{L\Delta}{T+1}}.
\]
\end{proof}

\section{Upper Bounds via Scalarization and Rate Comparison}\label{sec:upper}
The lower bounds can be matched by applying an optimal single-objective first-order method to a fixed scalarization of the multiobjective problem. Throughout this section, for any fixed $\lambda\in\Delta^m$ we consider the scalarized objective $f_\lambda(x):=\sum_{i=1}^m \lambda_i f_i(x)$, which is $L$-smooth (and $\mu$-strongly convex when stated). We write $x_\lambda^\star\in\arg\min f_\lambda$ and $R_\lambda:=\|x^{(0)}-x_\lambda^\star\|$.

To establish upper bounds that match the orders of our lower bounds, we now analyze the performance of an optimal first-order method applied to a fixed scalarization $f_{\lambda}$. We recall the celebrated accelerated gradient method (AGD) of Nesterov. The convergence guarantees for AGD are typically stated in terms of the objective function suboptimality, $f(x) - f^*$. Lemma~\ref{lem:descent} provides the crucial link needed to translate these guarantees into bounds on the gradient norm.

\smallskip
\begin{lemma}[Descent lemma for the gradient norm]\label{lem:descent}
If $f$ is an $L$-smooth convex function with minimum value $f^\star$, then for all $x$,
\[
\|\nabla f(x)\|^2 \;\le\; 2L\big(f(x)-f^\star\big).
\]
\end{lemma}

\smallskip
\begin{definition}[AGD on a fixed scalarization]\label{def:agd}
Let $f_\lambda$ be $L$-smooth (and $\mu$-strongly convex when stated). We use the standard Nesterov accelerated gradient method:
\begin{itemize}\itemsep4pt
  \item \emph{Convex case ($\mu=0$).} Initialize $x^0=y^0\in\R^d$ and $t_0=1$. For $k=0,1,\dots,T-1$,
  \begin{align*}
  x^{k+1} &= y^k - \tfrac{1}{L}\,\nabla f_\lambda(y^k), \\
  t_{k+1} &= \tfrac{1+\sqrt{1+4t_k^{\,2}}}{2}, \\
  y^{k+1} &= x^{k+1} + \tfrac{t_k-1}{t_{k+1}}\bigl(x^{k+1}-x^k\bigr).
  \end{align*}
  \item \emph{Strongly convex case ($\mu>0$).} Let $q:=\sqrt{\mu/L}$ and $\beta:=\tfrac{1-q}{1+q}$. Initialize $x^0=y^0\in\R^d$. For $k=0,1,\dots,T-1$,
  \begin{align*}
  x^{k+1} &= y^k - \tfrac{1}{L}\,\nabla f_\lambda(y^k), \\
  y^{k+1} &= x^{k+1} + \beta\bigl(x^{k+1}-x^k\bigr).
  \end{align*}
\end{itemize}
\end{definition}

\smallskip
The convergence properties of this algorithm are well-established and provide optimal rates for the class of first-order methods on smooth convex and smooth, strongly convex functions. We summarize these classical results in the following lemma for completeness.

\smallskip
\begin{lemma}[AGD convergence (see \cite{Nesterov04}, \S2.2; \cite{Bubeck2015}, \S3.7.2; \cite{BeckTeboulle2009}]\label{lem:agd-rate}
Let $f$ be $L$-smooth and convex with minimizer $x^\star$ and $R:=\|x^{(0)}-x^\star\|$. The AGD iterates from Definition~\ref{def:agd} satisfy
\[
f\big(x^{(T)}\big)-f(x^\star) \;\le\; \frac{2L\,R^2}{(T+1)^2}.
\]
If, in addition, $f$ is $\mu$-strongly convex with condition number $\kappa:=L/\mu$, then
\[
f\big(x^{(T)}\big)-f(x^\star) \;\le\; \frac{L+\mu}{2}\,R^2\,
\Bigg(\frac{\sqrt{\kappa}-1}{\sqrt{\kappa}+1}\Bigg)^{\!2T}.
\]
\end{lemma}

\smallskip
With these classical convergence results in hand, we can now establish upper bounds for the Pareto stationarity gap. The strategy is straightforward: we apply AGD to a fixed scalarization $f_{\lambda}$ and leverage the fact that, by definition, the Pareto stationarity gap $\cG(x)$ is bounded above by the norm of the gradient of any scalarized objective, i.e., $\cG(x) \le ||\nabla f_{\lambda}(x)||$. This yields a direct path from the scalar convergence rates to worst-case guarantees for the multiobjective problem.

\smallskip
\begin{theorem}[Upper bounds via AGD on a fixed scalarization]\label{thm:upper-bound-agd}
Assume each $f_i$ is convex and $L$-smooth. Fix any $\lambda\in\Delta^m$ and run AGD (Definition~\ref{def:agd}) on $f_\lambda$ for $T$ steps from $x^{(0)}$. Then
\[
\mathcal{G}\big(x^{(T)}\big) \;\le\; \|\nabla f_\lambda(x^{(T)})\| \;\le\; \frac{2LR_\lambda}{T+1}.
\]
If, moreover, each $f_i$ is $\mu$-strongly convex (so $f_\lambda$ is $\mu$-strongly convex with $\kappa=L/\mu$), then
\[
\mathcal{G}\big(x^{(T)}\big) \;\le\; \|\nabla f_\lambda(x^{(T)})\|
\;\le\; \sqrt{\,L(L+\mu)\,}\;R_\lambda\;
\Bigg(\frac{\sqrt{\kappa}-1}{\sqrt{\kappa}+1}\Bigg)^{\!T}.
\]
\end{theorem}

\begin{proof}
For any $x$, $\mathcal{G}(x)\le \|\nabla f_\lambda(x)\|$ by definition of the Pareto stationarity gap. In the convex case, combine Lemma~\ref{lem:descent} with Lemma~\ref{lem:agd-rate} applied to $f_\lambda$:
\[
\|\nabla f_\lambda(x^{(T)})\|^2 \;\le\; 2L\big(f_\lambda(x^{(T)})-f_\lambda(x_\lambda^\star)\big)
\;\le\; 2L\cdot \frac{2L R_\lambda^2}{(T+1)^2}
\;=\; \frac{4L^2 R_\lambda^2}{(T+1)^2},
\]
and take square roots. In the strongly convex case, apply the same inequality $\|\nabla f_\lambda(x)\|^2 \le 2L(f_\lambda(x)-f_\lambda^\star)$ together with the strongly convex part of Lemma~\ref{lem:agd-rate}:
\[
\|\nabla f_\lambda(x^{(T)})\|^2 \;\le\; 2L\cdot \frac{L+\mu}{2}\,R_\lambda^2\,
\Bigg(\frac{\sqrt{\kappa}-1}{\sqrt{\kappa}+1}\Bigg)^{\!2T}
\;=\; L(L+\mu)\,R_\lambda^2\,
\Bigg(\frac{\sqrt{\kappa}-1}{\sqrt{\kappa}+1}\Bigg)^{\!2T},
\]
and take square roots to obtain the stated bound. Note that AGD is an oblivious span method (but not an oblivious one-step method). Hence an $\mathcal{O}(1/T)$ upper bound in the convex case is achievable by an \emph{oblivious span} method on a fixed scalarization; the linear bound in the strongly convex case follows analogously with the classical accelerated factor.
\end{proof}

\smallskip
\begin{corollary}[AGD $\varepsilon$-complexity for the Pareto gap]\label{cor:agd-eps}
Under the assumptions of Theorem~\ref{thm:upper-bound-agd} (convex case), to guarantee $\mathcal{G}(x^{(T)})\le \varepsilon$ it suffices to take
\[
T \ \ge\ \left\lceil \frac{2 L R_\lambda}{\varepsilon} \right\rceil - 1.
\]
In the $\mu$-strongly convex case with $\kappa=L/\mu$, it suffices to take
\[
T \ \ge\ \left\lceil \frac{1}{\ln\!\big(\tfrac{\sqrt{\kappa}+1}{\sqrt{\kappa}-1}\big)}\;
\ln\!\Big(\frac{\sqrt{L(L+\mu)}\,R_\lambda}{\varepsilon}\Big)\right\rceil,
\]
that is, $T=\Theta\!\big(\sqrt{\kappa}\,\log(\sqrt{L(L+\mu)}\,R_\lambda/\varepsilon)\big)$.
\end{corollary}

\smallskip
\begin{remark}[Tighter scalar upper bounds via gradient-minimization methods]\label{rem:lan-oyz}
The bounds in Theorem~5.4--Corollary~5.5 are obtained by applying AGD to a fixed
scalarization $f_\lambda$ and then translating function-value convergence into a gradient-norm
bound via Lemma~5.1 (using that $G(x)\le \|\nabla f_\lambda(x)\|$).
If one is specifically interested in \emph{gradient-norm} complexity for the scalar problem
$\min_x f_\lambda(x)$, then tighter results are available in the single-objective literature.

In particular, Lan, Ouyang and Zhang~\cite{lan2023gradientmin} develop optimal (and
parameter-free) first-order methods for computing $\hat x$ with small (projected) gradient norm.
Specializing to our unconstrained setting (where projected gradient reduces to the gradient),
their results imply that for convex $L$-smooth $f_\lambda$ one can obtain
$\|\nabla f_\lambda(\hat x)\|\le \varepsilon$ within $\mathcal{O}\!\left(\sqrt{L}\,R_\lambda/\varepsilon\right)$
gradient evaluations, where $R_\lambda=\|x^{(0)}-x^\star_\lambda\|$.
In the $\mu$-strongly convex case, they obtain
$\mathcal{O}\!\left(\sqrt{L/\mu}\,\log(\|\nabla f_\lambda(x^{(0)})\|/\varepsilon)\right)$
gradient evaluations.
Combining these guarantees with $G(x)\le \|\nabla f_\lambda(x)\|$ yields the corresponding
Pareto-gap upper bounds for the multiobjective problem when optimizing a fixed scalarization.
\end{remark}

\smallskip
\begin{remark}[Best known one-step upper bound]\label{rem:onestep}
For $L$-smooth convex $f_\lambda$ and stepsizes $\alpha_t\le 1/L$, gradient descent satisfies
\[
\sum_{t=0}^{T-1}\|\nabla f_\lambda(x^{(t)})\|^2 \;\le\; 2L\big(f_\lambda(x^{(0)})-f_\lambda^\star\big)
\;\le\; L^2 R_\lambda^2,
\]
whence $\min_{0\le t<T}\|\nabla f_\lambda(x^{(t)})\|\le LR_\lambda/\sqrt{T}$. Since $\mathcal{G}(x)\le \|\nabla f_\lambda(x)\|$, we obtain $\min_{0\le t<T}\mathcal{G}(x^{(t)})\le LR_\lambda/\sqrt{T}$ for the oblivious one-step class.
\end{remark}

\smallskip
\begin{remark}[Instance calibration and rate comparison]\label{rem:calibration}
The upper bounds above depend on $R_\lambda=\|x^{(0)}-x_\lambda^\star\|$, whereas our lower bounds depend on $R=\mathrm{dist}(x^{(0)},\mathcal{P})$. Always $R\le R_\lambda$. For the hard instances in Theorem~\ref{thm:lifting}, choosing $\lambda=e_1$ yields $R_{e_1}=R$, so the orders match. Importantly, the $\Omega(1/T)$ lower bound in Theorem~\ref{thm:oblivious-lower}$\,$holds for the \emph{oblivious one-step} class, whereas the $\mathcal{O}(1/T)$ upper bound here is achieved by an \emph{oblivious span} method (AGD). Whether $\mathcal{O}(1/T)$ is achievable within the oblivious one-step class remains open; the best known upper bound there is $\mathcal{O}(LR/\sqrt{T})$.
\end{remark}

\section{Conclusion}\label{sec:conclusions}

We developed an oracle-complexity theory for finding $\varepsilon$-Pareto stationary points in smooth multiobjective optimization, measured by the Pareto stationarity gap $\cG(x)$. A central tool is a robust non-degenerate lifting construction that embeds scalar hard instances into MOO instances with distinct objectives and non-singleton Pareto fronts while preserving lower bounds on $\cG$. 

In the $\mu$-strongly convex regime we obtained a tight linear-rate lower bound for span methods matching accelerated upper bounds. In the convex regime we separated algorithmic power: oblivious one-step methods admit an $\Omega(1/T)$ min-iterate lower bound, while oblivious span methods satisfy a universal polynomial-degree $\Omega(1/T^2)$ last-iterate bound. We then showed that the polynomial-degree limitation is loose: by lifting the geometric stationary-point lower bounds of Carmon--Duchi--Hinder--Sidford, we proved an $\Omega(\sqrt{L\Delta}/T)$ min-iterate lower bound for adaptive first-order methods under bounded scalarization gap, recovering $\Omega(LR/T)$ on our lifted family. Finally, leveraging that the lifting does not require convexity of the embedded function, we extended the framework to nonconvex MOO under Lipschitz gradients, obtaining the tight-in-order $\Omega(\sqrt{L\Delta/T})$ complexity.

Several directions remain open. Most prominently, the last-iterate complexity of adaptive first-order methods in convex MOO under bounded scalarization gap (or under geometric distance models) is not yet fully characterized. Extending the framework to randomized methods, stochastic oracles, constraints/composite objectives, and refined dimension-dependent regimes are further natural directions for a comprehensive complexity theory of Pareto stationarity.

\appendix
\renewcommand{\thesection}{\Alph{section}} 
\renewcommand{\thesubsection}{\thesection.\arabic{subsection}} 
\renewcommand{\thetheorem}{\thesubsection.\arabic{theorem}} 
\makeatletter
\@addtoreset{theorem}{subsection}
\makeatother

\section{Derivations of Polynomial Extremal Results}

\subsection{A tight \texorpdfstring{$\Omega(1/T)$}{Omega(1/T)} lower bound for oblivious one-step methods (with $L$-capped steps)}\label{app:prod-poly-tight}

In this appendix we establish a tight $\Omega(1/T)$ lower bound for the stationarity\hyp{}gap proxy
\[
\max_{\zeta\in[0,L]}\;\zeta\;\prod_{k=0}^{t-1}\bigl(1-\alpha_k\,\zeta\bigr),
\]
which arises in the analysis of oblivious one-step methods on $L$-smooth convex quadratics. The proof is self-contained and uses only elementary inequalities.

\paragraph{Model and class}
We consider \emph{oblivious one-step gradient methods} applied to a fixed scalarization $g(x)=\tfrac12 x^\top H x - b^\top x$ of a multiobjective instance, where $H\succeq 0$ and $\sigma(H)\subset[0,L]$.
The method generates
\[
x^{(t+1)} \;=\; x^{(t)} - \alpha_t \,\nabla g\bigl(x^{(t)}\bigr),\qquad t=0,1,2,\dots,
\]
with a \emph{pre-scheduled} stepsize sequence $\{\alpha_t\}_{t\ge 0}$ which may depend on known problem data (here, $L$) but not on oracle feedback. We adopt the natural $L$-cap:
\begin{equation}\label{eq:alpha-cap}
0 \;\le\; \alpha_t \;\le\; \frac{1}{L}\qquad\text{for all }t\ge 0.
\end{equation}
This is the standard safe region for gradient descent on $L$-smooth convex quadratics with known $L$ and is the relevant oblivious regime for worst-case analysis.

\paragraph{Polynomial representation}
Let $e^{(t)}=x^{(t)}-x^\star$, where $x^\star\in\arg\min g$; then for quadratics $e^{(t)}=p_t(H)\,e^{(0)}$ with
\[
p_t(\zeta)\;=\;\prod_{k=0}^{t-1}\bigl(1-\alpha_k\,\zeta\bigr),\qquad p_0\equiv 1,
\]
and hence $\nabla g\bigl(x^{(t)}\bigr)=H\,p_t(H)\,e^{(0)}$ (Lemma~\ref{lem:krylov}).

We prove the following extremal result.

\begin{theorem}[Tight $\Omega(1/T)$ product-form extremal under $L$-capped steps]\label{thm:prod-omega-one-over-T}
Fix $L>0$ and $t\in\mathbb{N}$. For any nonnegative stepsizes $\alpha_0,\ldots,\alpha_{t-1}$ satisfying \eqref{eq:alpha-cap}, we have
\begin{equation}\label{eq:prod-lb}
\max_{\zeta\in[0,L]}\;\zeta\;\prod_{k=0}^{t-1}\bigl(1-\alpha_k\,\zeta\bigr)
\;\;\ge\;\; \frac{L}{4\,(t+1)}.
\end{equation}
Moreover, the bound is tight up to a universal constant: the constant stepsizes $\alpha_k\equiv 1/L$ yield
\begin{equation}\label{eq:prod-ub}
\max_{\zeta\in[0,L]}\;\zeta\;\prod_{k=0}^{t-1}\bigl(1-\alpha_k\,\zeta\bigr)
\;=\; \max_{\zeta\in[0,L]}\;\zeta\Bigl(1-\frac{\zeta}{L}\Bigr)^{t}
\;\;\le\;\; \frac{L}{e\,(t+1)}.
\end{equation}
\end{theorem}

\begin{remark}[Scope of \eqref{eq:alpha-cap}]
The cap \eqref{eq:alpha-cap} is standard in first-order analysis with known $L$. Allowing stepsizes larger than $1/L$ may cause ascent on $L$-smooth convex quadratics and does not help the algorithm in the worst case. Our lower bound therefore covers the canonical oblivious regime.
\end{remark}

The proof rests on a simple but sharp product inequality.

\begin{lemma}[A product lower bound]\label{lem:prod-ineq}
For any $r\in\mathbb{N}$ and any numbers $x_1,\ldots,x_r\in[0,1]$,
\begin{equation}\label{eq:prod-ineq}
\prod_{i=1}^{r}\bigl(1-x_i\bigr)\;\;\ge\;\; 1-\sum_{i=1}^{r} x_i.
\end{equation}
\end{lemma}

\begin{proof}
By induction on $r$. The case $r=1$ is trivial. Suppose \eqref{eq:prod-ineq} holds for $r-1$. Then for $r$ we have
\begin{align*}
\prod_{i=1}^{r}(1-x_i)
  &= \Bigl(\prod_{i=1}^{r-1}(1-x_i)\Bigr)(1-x_r) \\
  &\ge \Bigl(1-\sum_{i=1}^{r-1}x_i\Bigr)(1-x_r) \\
  &= 1-\sum_{i=1}^r x_i \;+\; x_r\sum_{i=1}^{r-1}x_i \\
  &\ge 1-\sum_{i=1}^r x_i.
\end{align*}
since $x_r\sum_{i=1}^{r-1}x_i\ge 0$.
\end{proof}

\begin{proof}[Proof of Theorem~\ref{thm:prod-omega-one-over-T}]
Fix $t \ge 1$ and a stepsize sequence $(\alpha_k)_{k=0}^{t-1}$ with $0 \le \alpha_k \le \tfrac{1}{L}$ for all $k$.
For any $\zeta\in[0,L]$, set $x_k:=\alpha_k\,\zeta\in[0,1]$. Applying Lemma~\ref{lem:prod-ineq} gives
\begin{equation}\label{eq:prod-linear}
\prod_{k=0}^{t-1}\bigl(1-\alpha_k\,\zeta\bigr)
\;\;\ge\;\; 1 - \zeta\sum_{k=0}^{t-1}\alpha_k.
\end{equation}
Define the partial sums $S_t := \sum_{k=0}^{t-1}\alpha_k$. Note that $S_t\le t\cdot (1/L) = t/L$ by \eqref{eq:alpha-cap}.

Consider the function $\Phi_t(\zeta):=\zeta\prod_{k=0}^{t-1}(1-\alpha_k\,\zeta)$ on $[0,L]$. By \eqref{eq:prod-linear},
\begin{equation}\label{eq:phi-lb}
\Phi_t(\zeta) \;\ge\; \zeta\bigl(1-\zeta S_t\bigr)
\qquad\text{for all }\zeta\in[0,L].
\end{equation}
Let $\zeta^\star := \min\{L,\; 1/(2S_t)\}$ (with the convention $1/0=+\infty$). We distinguish two exhaustive cases.

\smallskip
\emph{Case 1: $S_t \ge 1/(2L)$.} Then $\zeta^\star=1/(2S_t)\le L$, and hence, by \eqref{eq:phi-lb},
\[
\Phi_t(\zeta^\star) \;\ge\; \zeta^\star\Bigl(1-\zeta^\star S_t\Bigr)
\;=\; \frac{1}{2S_t}\cdot\Bigl(1-\tfrac12\Bigr)
\;=\; \frac{1}{4S_t}
\;\ge\; \frac{1}{4\,(t/L)} \;=\; \frac{L}{4t}.
\]

\smallskip
\emph{Case 2: $S_t < 1/(2L)$.} Then $\zeta^\star=L$ and, by \eqref{eq:phi-lb},
\[
\Phi_t(\zeta^\star) \;\ge\; L\bigl(1-L S_t\bigr)
\;\ge\; L\Bigl(1-\tfrac12\Bigr)
\;=\; \frac{L}{2}
\;\ge\; \frac{L}{4t}\qquad(\text{since }t\ge 1).
\]
In both cases, we obtain $\max_{\zeta\in[0,L]}\Phi_t(\zeta)\ge L/(4t)$. Since $t \ge 1$ implies $1/t \ge 1/(t+1)$, this establishes the slightly looser bound in \eqref{eq:prod-lb}.

For the near-matching \emph{upper} bound \eqref{eq:prod-ub}, take the constant stepsizes $\alpha_k\equiv 1/L$. Then
\[
\Psi_t(\zeta):=\zeta\Bigl(1-\frac{\zeta}{L}\Bigr)^{t},\qquad \zeta\in[0,L].
\]
Elementary calculus shows that $\Psi_t$ attains its maximum at $\zeta^\diamond=L/(t+1)$, with
\begin{align*}
\max_{\zeta\in[0,L]}\Psi_t(\zeta)
  &= \Psi_t(\zeta^\diamond) \\
  &= \frac{L}{t+1}\left(1-\frac{1}{t+1}\right)^{t} \\
  &\le \frac{L}{t+1}\,\mathrm{e}^{-\,t/(t+1)} \\
  &\le \frac{L}{\mathrm{e}\,(t+1)}.
\end{align*}
This establishes \eqref{eq:prod-ub}.
\end{proof}

We now translate Theorem~\ref{thm:prod-omega-one-over-T} into an oracle lower bound for the gradient norm.

\begin{corollary}[Oblivious one-step lower bound for the gradient norm]\label{cor:oblivious-gradient-lb}
Let $t\in\mathbb{N}$ and let an oblivious one-step method with stepsizes $\alpha_0,\ldots,\alpha_{t-1}$ obeying \eqref{eq:alpha-cap} be given. Then there exists an $L$-smooth convex quadratic $g(x)=\tfrac12 x^\top H x - b^\top x$ with $\sigma(H)\subset[0,L]$ and an initialization $x^{(0)}$ satisfying $\|x^{(0)}-x^\star\|=R$ such that
\[
\min_{0\le s\le t}\;\bigl\|\nabla g\bigl(x^{(s)}\bigr)\bigr\|
\;\;\ge\;\; \frac{L\,R}{4\,(t+1)}.
\]
\end{corollary}

\begin{proof}
For quadratics, $\|\nabla g(x^{(s)})\|=\|H\,p_s(H)\,(x^{(0)}-x^\star)\|$. Fix any unit vector $u$ and take $x^{(0)}-x^\star=R\,u$.
For every fixed $\zeta\in[0,L]$, the monotonicity $(1-\alpha_s\zeta)\in[0,1]$ implies
\[
\zeta\,\Bigl|\prod_{k=0}^{s-1}\bigl(1-\alpha_k\zeta\bigr)\Bigr|
\;\;\text{is nonincreasing in $s$.}
\]
Hence, for any fixed $\zeta$,
\[
\min_{0\le s\le t}\ \zeta\,\Bigl|\prod_{k=0}^{s-1}\bigl(1-\alpha_k\zeta\bigr)\Bigr|
\;=\; \zeta\,\prod_{k=0}^{t-1}\bigl(1-\alpha_k\zeta\bigr).
\]
Choosing $H$ to be diagonal with an eigenvalue at the maximizer $\zeta^\star\in[0,L]$ of the right-hand side and aligning $u$ with the corresponding eigenvector yields
\[
\min_{0\le s\le t}\;\|\nabla g(x^{(s)})\|
\;=\; R\cdot \max_{\zeta\in[0,L]}\;\zeta\;\prod_{k=0}^{t-1}\bigl(1-\alpha_k\,\zeta\bigr)
\;\ge\; \frac{L\,R}{4\,(t+1)},
\]
by Theorem~\ref{thm:prod-omega-one-over-T}.
\end{proof}

This result is lifted to the multiobjective setting using Theorem~\ref{thm:lifting}: for the instance $f_i(x)=g(x_V)+\frac{L}{2}\|x_W-a_i\|^2$, we have $\cG(x)\ge \|\nabla g(x_V)\|$ and $\mathrm{dist}(x^{(0)},\mathcal{P})=R$. Corollary~\ref{cor:oblivious-gradient-lb} thus directly implies a lower bound on $\min_{s\le t}\cG(x^{(s)})$. This establishes the bound for a non-degenerate MOO instance with distinct objectives and a non-singleton Pareto set.

\subsection{Universal Markov-Based Bound for Oblivious Span Methods}\label{app:markov}
The lower bound for oblivious methods relies on $\inf_{p_t} \max_{\zeta \in [0,L]} |\zeta p_t(\zeta)|$ where $p_t$ is any polynomial with $\deg p_t \le t$ and $p_t(0)=1$.
\begin{lemma}[Markov Brothers' Inequality (see \cite{Rivlin74}, Chapter 2, Section 2.7, Page 123)]
For any polynomial $q(x)$ of degree $n$,
\[
\max_{x\in[-1,1]} |q'(x)| \le n^2 \max_{x\in[-1,1]} |q(x)|.
\]
\end{lemma}
Let $S(\zeta) = \zeta p_t(\zeta)$, which is a polynomial of degree $t+1$. We know $S(0)=0$ and from $p_t(0)=1$, we have $S'(0) = p_t(0) + 0 \cdot p_t'(0) = 1$. Let's map the interval $[0,L]$ to $[-1,1]$ via the transformation $x=\frac{2\zeta}{L}-1$, which means $\zeta=\frac{L(x+1)}{2}$. Define a new polynomial $q(x) = S\left(\frac{L(x+1)}{2}\right)$. The degree of $q$ is $n=t+1$. At $x=-1$, $q(-1) = S(0)=0$.
The derivative of $q$ with respect to $x$ is $q'(x) = S'\left(\frac{L(x+1)}{2}\right) \cdot \frac{L}{2}$.
At $x=-1$, this gives $q'(-1) = S'(0) \cdot \frac{L}{2} = 1 \cdot \frac{L}{2} = \frac{L}{2}$.
By Markov's inequality:
\[
\frac{L}{2} = |q'(-1)| \le \max_{x\in[-1,1]}|q'(x)| \le (t+1)^2 \max_{x\in[-1,1]} |q(x)|.
\]
Rearranging gives the bound on the maximum of $q$, which is the same as the maximum of $S$:
\[
\max_{\zeta\in[0,L]} |\zeta p_t(\zeta)| = \max_{x\in[-1,1]} |q(x)| \ge \frac{L}{2(t+1)^2}.
\]

\subsection{Chebyshev Extremal Problem on \texorpdfstring{$[\mu,L]$}{[mu,L]}}\label{app:cheb-strong}
The lower bound for strongly convex methods requires solving
\[
\min_{p_t}\; \max_{\zeta\in[\mu,L]} \lvert p_t(\zeta)\rvert
\quad \text{s.t.}\quad \deg p_t \le t,\; p_t(0)=1 .
\]
The affine transformation $\xi = \frac{2\zeta - (L+\mu)}{L-\mu}$ maps the interval $\zeta \in [\mu, L]$ to $\xi \in [-1,1]$. The constraint $p_t(0)=1$ is evaluated at $\zeta=0$, which corresponds to $\xi_0 = -\frac{L+\mu}{L-\mu}$. Let $\tilde{p}_t(\xi) = p_t(\zeta(\xi))$. The problem becomes:
\[
\min_{\substack{\deg \tilde{p}_t \le t \\ \tilde{p}_t(\xi_0)=1}} \max_{\xi \in [-1, 1]} |\tilde{p}_t(\xi)|.
\]
The solution to this classic problem is the scaled Chebyshev polynomial $\tilde{p}_t^*(\xi) = \frac{T_t(\xi)}{T_t(\xi_0)}$, where $T_t$ is the Chebyshev polynomial of the first kind of degree $t$ (see \cite{Rivlin74}, Chapter 2, Section 2.7, Page 108). The extremal value is $\frac{1}{|T_t(\xi_0)|}$ because $\max_{\xi\in[-1,1]}|T_t(\xi)|=1$.
With $\kappa=L/\mu$, we have $|\xi_0|=\frac{L+\mu}{L-\mu}=\frac{\kappa+1}{\kappa-1} > 1$.
For $|x|>1$, the Chebyshev polynomial is given by $T_t(x) = \frac{1}{2}\left( (x+\sqrt{x^2-1})^t + (x-\sqrt{x^2-1})^{-t} \right)$.
Let $\rho = |\xi_0|+\sqrt{|\xi_0|^2-1}$. A direct calculation shows:
\begin{align*}
\rho 
  &= \frac{\kappa+1}{\kappa-1} 
     + \sqrt{\left(\frac{\kappa+1}{\kappa-1}\right)^2 - 1} \\[6pt]
  &= \frac{\kappa+1+2\sqrt{\kappa}}{\kappa-1} \\[6pt]
  &= \frac{(\sqrt{\kappa}+1)^2}{(\sqrt{\kappa}-1)(\sqrt{\kappa}+1)} \\[6pt]
  &= \frac{\sqrt{\kappa}+1}{\sqrt{\kappa}-1}.
\end{align*}
Then $|T_t(\xi_0)| = |T_t(|\xi_0|)| = \frac{1}{2}(\rho^t+\rho^{-t})$. The extremal value is $\frac{1}{|T_t(\xi_0)|} = \frac{2}{\rho^t+\rho^{-t}}$.

\bibliographystyle{plain}
\bibliography{refs}

\end{document}